\documentclass{article}
\usepackage{geometry}
\usepackage[utf8]{inputenc}
\usepackage{caption}
\newcommand{\nocaption}{\captionsetup{labelsep=none}\caption{}}

\providecommand{\keywords}[1]
{
  \small	
  \textbf{\textit{Keywords:}} #1
}

\usepackage{orcidlink}

\newcommand{\ZFC}{\mathrm{ZFC}}
\newcommand{\BZFC}{\mathrm{BZFC}}

\newcommand{\vt}{\ensuremath{\mathbf{t}}}
\newcommand{\vf}{\ensuremath{\mathbf{f}}}
\newcommand{\vb}{\ensuremath{\mathbf{b}}}
\newcommand{\vn}{\ensuremath{\mathbf{n}}}
\newcommand{\BS}{\mathrm{BS4}}
\newcommand{\bb}{\mathfrak{b}}
\newcommand{\nn}{\mathfrak{n}}
\newcommand{\rlm}{\mathrm{rlm}}
 \newcommand{\no}{{\sim}} 
 \newcommand{\abbr}{\;\;\; \text{abbreviates} \;\;\; }

\makeatletter
\DeclareOldFontCommand{\rm}{\normalfont\rmfamily}{\mathrm}
\DeclareOldFontCommand{\sf}{\normalfont\sffamily}{\mathsf}
\DeclareOldFontCommand{\tt}{\normalfont\ttfamily}{\mathtt}
\DeclareOldFontCommand{\bf}{\normalfont\bfseries}{\mathbf}
\DeclareOldFontCommand{\it}{\normalfont\itshape}{\mathit}
\DeclareOldFontCommand{\sl}{\normalfont\slshape}{\@nomath\sl}
\DeclareOldFontCommand{\sc}{\normalfont\scshape}{\@nomath\sc}
\makeatother

\usepackage{enumerate}

\usepackage{pgf,tikz,pgfplots}
\usepackage{tikz-cd}
\usepackage{tikz-3dplot}

\pgfplotsset{compat=1.15}
\usepackage{mathrsfs}
\usetikzlibrary{arrows}
\usetikzlibrary{arrows.meta}

\usepackage{amsmath}
\usepackage{amssymb}
\usepackage{yfonts}
\usepackage{amsthm}

\usepackage{multicol}

\usepackage{float}

\usepackage{mdframed}

\usepackage{hyperref}

\usepackage[toc,page]{appendix}

\newcommand{\llbracket}{[\![}
\newcommand{\rrbracket}{]\!]}

\newtheorem{theorem}{Theorem}[section]
\newtheorem{corollary}[theorem]{Corollary}

\newtheorem{proposition}[theorem]{Proposition}

\theoremstyle{definition}
\newtheorem*{remark}{Remark}

\theoremstyle{definition}
\newtheorem{definition}[theorem]{Definition}

\theoremstyle{definition}
\newtheorem{example}[theorem]{Example}

\usepackage{mathtools}

\makeatletter
\DeclareRobustCommand\widecheck[1]{{\mathpalette\@widecheck{#1}}}
\def\@widecheck#1#2{%
    \setbox\z@\hbox{\m@th$#1#2$}%
    \setbox\tw@\hbox{\m@th$#1%
       \widehat{%
          \vrule\@width\z@\@height\ht\z@
          \vrule\@height\z@\@width\wd\z@}$}%
    \dp\tw@-\ht\z@
    \@tempdima\ht\z@ \advance\@tempdima2\ht\tw@ \divide\@tempdima\thr@@
    \setbox\tw@\hbox{%
       \raise\@tempdima\hbox{\scalebox{1}[-1]{\lower\@tempdima\box
\tw@}}}%
    {\ooalign{\box\tw@ \cr \box\z@}}}
\makeatother

\usepackage{url}

\title{Cardinality in a paraconsistent and paracomplete set theory}
\author{Hrafn Valtýr Oddsson\thanks{Institut für Philosophie I, Ruhr-Universit\"at Bochum. Email: hrafn.oddsson@rub.de} \ \orcidlink{0000-0003-1594-340X}}
\date{}

\begin{document}
\maketitle

\begin{abstract}
This paper develops a rich theory of cardinality in the paraconsistent and paracomplete set theory $\BZFC$, where sets can be inconsistent ($A$ such that ``$x\in A$'' is both true and false for some $x$) or incomplete ($A$ such that ``$x\in A$'' is neither true nor false for some $x$). We carefully analyze what it means for two potentially incomplete or inconsistent sets to have ``the same size'', construct the corresponding cardinal numbers, and develop the basic theory of cardinal arithmetic. A surprising result is that the cardinality of any set can be expressed as a linear combination of three fundamental cardinal numbers with classical cardinals as coefficients. In that sense, our cardinal numbers form a three-dimensional space over the usual cardinals, much like how the complex numbers form a two-dimensional space over the reals.

\end{abstract}


\section{Introduction}
In \cite{KHOMSKII_ODDSSON_2024,HrafnThesis}, a formalization of paraconsistent and paracomplete set theory called $\BZFC$ was developed in the logic $\BS$ from \cite{BS4}. In the semantics of $\BS$, truth and falsity are separated, so a statement $\varphi$ can be true and not false ($\vt$), false and not true ($\vf$), both true and false ($\vb$), or neither true nor false ($\vn$). Accordingly, a set $A$ is called \emph{inconsistent} if the statement ``$x\in A$'' gets the truth value $\vb$ for some $x$. It is called \emph{incomplete} if the statement gets the truth value $\vn$ for some $x$. A set is called \emph{classical} if it is neither inconsistent nor incomplete.

\medskip 

The question this paper aims to answer is: \emph{How do we measure the size of sets that are inconsistent or incomplete?} Consider, for example, a set $A$ with a unique element $a$ such that ``$a \in A$'' is true, while for \emph{all} $x$, including $a$, ``$x \in A$'' is false.  Then $A$ is inconsistent, as ``$a \in A$'' is both true and false. How many elements does $A$ have? On the one hand, we could try to say that $A$ has \emph{zero} elements since everything is a non-member of $A$. However, this fails to capture that $A$ has an element, namely $a$. On the other hand, we could try to say that the number of elements in $A$ is \emph{one} since $a$ is the unique element of $A$. But this fails to capture the fact that for all $x$, ``$x\in A$'' is false. So, both one and zero fail to adequately capture the size of $A$. We are forced to admit that we will need a new number that lies somewhere between zero and one to describe the size of $A$.  

\medskip

To address this, the paper develops a theory of cardinality suitable for $\BZFC$ by focusing on two closely related notions: First, \emph{equinumerosity} (denoted $A\cong B$), which establishes when two sets should be considered to have the same size. Second, the \emph{cardinal number} $|A|$ of a set $A$, representing its size. These two concepts must be aligned so that $A\cong B$ gets the same truth value as $|A|=|B|$. This is because two sets having the same size should mean the same thing as their sizes being equal.

\medskip

Of course, this is far from the first time that the notion of cardinality has been investigated in a paraconsistent or paracomplete setting. As an example, see \cite{Weber2021} for a treatment in a paraconsistent set theory with a full comprehension axiom. 
However, to my knowledge, no previous attempt can adequately capture the situation described above. 

\medskip

The paper is structured as follows. In Section~\ref{preliminaries}, we provide the necessary preliminaries. We review the logic $\BS$ and the set theory $\BZFC$. The aim is to keep the paper sufficiently self-contained so that a new reader can follow along. However, we will be rather brief and leave out many technical details that can be found in \cite{KHOMSKII_ODDSSON_2024, HrafnThesis}. In Section~\ref{Sec:An Informal Treatment of Cardinality}, we give an informal account of cardinality in a $\BS$-style setting. The aim of this section is to give an intuitive picture that motivates the formal definitions in the later sections. In Section~\ref{EquinumerositySection}, we formalize our notion of equinumerosity between sets in this non-classical setting that allows us to compare the sizes of inconsistent or incomplete sets. We establish the basic properties of this notion and show how it relates to the classical notion. Finally, Section~\ref{CardNumSect} covers the cardinal numbers corresponding to our notion of equinumerosity and the basic theory of cardinal arithmetic. The main result here is that any cardinal $\kappa$ is uniquely expressible as a linear combination
$$\kappa=\kappa_\vt+\kappa_\vb\cdot \bb+\kappa_\vn \cdot\nn.$$
Here, $\kappa_\vt$, $\kappa_\vb$, and $\kappa_\vn$ are classical cardinals, and $\bb$ and $\nn$ are fundamental non-classical cardinal numbers, representing the sizes of simple inconsistent and incomplete sets, respectively. This result shows that our cardinal numbers naturally form a three-dimensional space over the classical cardinals, analogous to how complex numbers form a two-dimensional space over the reals.

\section{Preliminaries}\label{preliminaries} 
Before developing our theory of cardinality, we need to establish the logical and set-theoretical foundations. In this section, we review the four-valued logic $\BS$ and the set theory $\BZFC$ built on it. The aim is to keep the paper self-contained. However, for the sake of brevity, we will leave out some technical details, which can be found in \cite{KHOMSKII_ODDSSON_2024,HrafnThesis}.

\subsection{The logic}

We start by reviewing the logic $\BS$ underlying the set theory $\BZFC$. The propositional fragment appeared under the name $\mathsf{CLoNs}$ in \cite{ExtraBS4}, and our version is essentially due to Omori and Waragai \cite{BS4}. 

For the sake of brevity, we will present the logic semantically and in a classical meta-theory only. However, it should be noted that $\BS$ has a simple sound and complete proof system (see \cite[Section 2.3]{KHOMSKII_ODDSSON_2024} for details, where it is slightly modified from \cite{BS4Semantics}). Moreover, $\BS$ can be given natural Tarski semantics in a paraconsistent and paracomplete meta-theory \cite[Section 9]{KHOMSKII_ODDSSON_2024}.  

\medskip

$\BS$ works by separating truth from falsity. If $\varphi$ is a sentence, then $\varphi$ may be true or not true in a model and, independently, false or not false in that same model. This is achieved by considering two satisfaction relations, $\vDash^+$ and $\vDash^-$, representing truth and falsity in a model, respectively.

\begin{definition}

A \emph{model} $\mathcal{M}$ consists of the following:
\begin{itemize}
    \item a non-empty domain $M$
    \item a constant $c^\mathcal{M}\in M$ for each constant symbol $c$
    \item a function $f^\mathcal{M}:M^n\to M$ for each $n$-ary function symbol $f$ 
    \item a \emph{positive} interpretation $R^+_\mathcal{M}\subseteq M^n$ and a \emph{negative} interpretation $R^-_\mathcal{M}\subseteq M^n$ for each $n$-ary relation symbol $R$
    \item a binary relation $=^+$ which coincides with the true equality relation, and a binary relation $=^-$ satisfying  $a =^- b$ iff $b =^- a$.
\end{itemize}

We recursively define the two satisfaction relations $\vDash^+$ and $\vDash^-$ as follows:
\begin{itemize}
    \item $\mathcal{M}\vDash^+R(a_1,...,a_n)  \; \text{ iff } \; (a_1,...,a_n)\in R^+_\mathcal{M}$ \\
    $\mathcal{M}\vDash^-R(a_1,...,a_n) \; \text{ iff } \; (a_1,...,a_n)\in R^-_\mathcal{M}$
    \item $\mathcal{M}\vDash^+\no\varphi  \; \text{ iff } \; \mathcal{M}\vDash^-\varphi$\\
    $\mathcal{M}\vDash^-\no\varphi  \; \text{ iff } \; \mathcal{M}\vDash^+\varphi$
    \item $\mathcal{M}\vDash^+\varphi\land\psi  \; \text{ iff } \; \mathcal{M}\vDash^+\varphi\text{ and }\mathcal{M}\vDash^+\psi$\\
     $\mathcal{M}\vDash^-\varphi\land\psi  \; \text{ iff } \; \mathcal{M}\vDash^-\varphi\text{ or }\mathcal{M}\vDash^-\psi$ 
    \item $\mathcal{M}\vDash^+\varphi\lor\psi  \; \text{ iff } \; \mathcal{M}\vDash^+\varphi\text{ or }\mathcal{M}\vDash^+\psi$\\
     $\mathcal{M}\vDash^-\varphi\lor\psi  \; \text{ iff } \; \mathcal{M}\vDash^-\varphi\text{ and }\mathcal{M}\vDash^-\psi$ 
    \item $\mathcal{M}\vDash^+\varphi\rightarrow\psi  \; \text{ iff } \; \mathcal{M}\vDash^+\varphi\text{ implies }\mathcal{M}\vDash^+\psi$\\
     $\mathcal{M}\vDash^-\varphi\rightarrow\psi  \; \text{ iff } \; \mathcal{M}\vDash^+\varphi\text{ and }\mathcal{M}\vDash^-\psi$
     \item $\mathcal{M}\vDash^+\varphi\leftrightarrow\psi  \; \text{ iff } \; \mathcal{M}\vDash^+\varphi\text{ iff }\mathcal{M}\vDash^+\psi$\\
     $\mathcal{M}\vDash^-\varphi\leftrightarrow\psi  \; \text{ iff } \; (\mathcal{M}\vDash^+\varphi\text{ and }\mathcal{M}\vDash^-\psi)\text{ or }(\mathcal{M}\vDash^-\varphi\text{ and }\mathcal{M}\vDash^+\psi)$
    \item $\mathcal{M}\vDash^+\forall x\varphi(x)  \; \text{ iff } \; \mathcal{M}\vDash^+\varphi(a)$ for all $a\in M$\\
    $\mathcal{M}\vDash^-\forall x\varphi(x)  \; \text{ iff } \; \mathcal{M}\vDash^-\varphi(a)$ for some $a\in M$
    \item $\mathcal{M}\vDash^+\exists x\varphi(x)  \; \text{ iff } \; \mathcal{M}\vDash^+\varphi(a)$ for some $a\in M$\\
    $\mathcal{M}\vDash^-\exists x\varphi(x)  \; \text{ iff } \; \mathcal{M}\vDash^-\varphi(a)$ for all $a\in M$
    \item $\mathcal{M}\nvDash^+\bot$\\
    $\mathcal{M}\vDash^-\bot$
\end{itemize}

If $\Gamma\cup\{\varphi\}$ is a set of sentences, then we write
$\Gamma\vDash\varphi$ iff for every model $\mathcal{M}$, $\mathcal{M}\vDash^+\Gamma$ implies $\mathcal{M}\vDash^+\varphi.$
    
\end{definition}
\medskip

We define the \emph{truth value} $\llbracket \varphi \rrbracket^\mathcal{M}$ of $\varphi$ in $\mathcal{M}$ as follows:
\begin{align*}   
\llbracket \varphi \rrbracket^\mathcal{M} &:=\begin{cases}
    $\vt$ & \text{if \ $\mathcal{M}\vDash^+\varphi$ and $\mathcal{M}\nvDash^-\varphi$} \\
    $\vb$ & \text{if \ $\mathcal{M}\vDash^+\varphi$ and $\mathcal{M}\vDash^-\varphi$}\\
   $\vn$ & \text{if \ $\mathcal{M}\nvDash^+\varphi$ and $\mathcal{M}\nvDash^-\varphi$}\\
    $\vf$ & \text{if \ $\mathcal{M}\nvDash^+\varphi$ and $\mathcal{M}\vDash^-\varphi$}
  \end{cases}
\end{align*}

\medskip

This gives the following truth tables:
\smallskip

    \begin{center}
\begin{tabular}{c|c}
&$\no$\\
\hline
$\vt$&$\vf$\\
$\vb$&$\vb$\\
$\vn$&$\vn$\\
$\vf$&$\vt$\\
\end{tabular}
\quad
\begin{tabular}{c|cccc}
$\land$&$\vt$&$\vb$&$\vn$&$\vf$\\
\hline
$\vt$&$\vt$&$\vb$&$\vn$&$\vf$\\
$\vb$&$\vb$&$\vb$&$\vf$&$\vf$\\
$\vn$&$\vn$&$\vf$&$\vn$&$\vf$\\
$\vf$&$\vf$&$\vf$&$\vf$&$\vf$\\
\end{tabular}
\quad
\begin{tabular}{c|cccc}
$\lor$&$\vt$&$\vb$&$\vn$&$\vf$\\
\hline
$\vt$&$\vt$&$\vt$&$\vt$&$\vt$\\
$\vb$&$\vt$&$\vb$&$\vt$&$\vb$\\
$\vn$&$\vt$&$\vt$&$\vn$&$\vn$\\
$\vf$&$\vt$&$\vb$&$\vn$&$\vf$\\
\end{tabular}

\medskip

\begin{tabular}{c|cccc}
$\rightarrow$&$\vt$&$\vb$&$\vn$&$\vf$\\
\hline
$\vt$&$\vt$&$\vb$&$\vn$&$\vf$\\
$\vb$&$\vt$&$\vb$&$\vn$&$\vf$\\
$\vn$&$\vt$&$\vt$&$\vt$&$\vt$\\
$\vf$&$\vt$&$\vt$&$\vt$&$\vt$\\
\end{tabular}
\quad
\begin{tabular}{c|cccc}
$\leftrightarrow$&$\vt$&$\vb$&$\vn$&$\vf$\\
\hline
$\vt$&$\vt$&$\vb$&$\vn$&$\vf$\\
$\vb$&$\vb$&$\vb$&$\vn$&$\vf$\\
$\vn$&$\vn$&$\vn$&$\vt$&$\vt$\\
$\vf$&$\vf$&$\vf$&$\vt$&$\vt$\\
\end{tabular}
\end{center}

\medskip

We say that $\varphi$ is \emph{classical} in $\mathcal{M}$ if $\llbracket \varphi\rrbracket^\mathcal{M}\in\{\vt,\vf\}.$ Similarly, we say that a sentence is \emph{consistent} in $\mathcal{M}$ if $\llbracket \varphi\rrbracket^\mathcal{M}\neq\vb$ and \emph{complete} if $\llbracket \varphi\rrbracket^\mathcal{M}\neq\vn$.

\subsection{Defined connectives}
We will need a few defined connectives. First, we note that a bi-implication $\varphi\leftrightarrow\psi$ tells us that $\varphi$ is true iff $\psi$ is true. However, we must be careful, since it does not tell us that $\varphi$ is false iff $\psi$ is false. So $\varphi\leftrightarrow\psi$ can be true without $\varphi$ and $\psi$ being interchangeable! Similarly,  $\varphi\rightarrow\psi$ tells us that $\varphi$ being true implies that $\psi$ is true. It does not tell us that $\psi$ being false implies that $\varphi$ is false. In other words, $\rightarrow$ does not contrapose with respect to $\no$. We, therefore, add the following abbreviations:
\begin{itemize}
\item $\varphi\Leftrightarrow\psi \abbr (\varphi\leftrightarrow \psi)\land (\no\varphi\leftrightarrow\no\psi)$ 
\item $\varphi\Rightarrow\psi \abbr (\varphi\rightarrow \psi)\land (\no\psi\rightarrow\no\varphi)$ 
\end{itemize}
These connectives are called \emph{strong equivalence} and \emph{strong implication}, respectively. 

\smallskip

Next, we note that $\no\varphi$ being true only allows us to conclude that $\varphi$ is false. It does not allow us to exclude the possibility that $\varphi$ is also true. To that end, we introduce the \emph{classical negation} $\neg$:
\begin{itemize} \item $\neg \varphi \abbr \varphi\rightarrow \bot$ 
\end{itemize}
The classical negation expresses the absence of truth and allows us to make the following abbreviations:

 \begin{itemize} \item $! \varphi \abbr \no \lnot \varphi $
\item $? \varphi \abbr \lnot \no  \varphi $
\item $\circ \varphi \abbr !\varphi\leftrightarrow ?\varphi$
\end{itemize}
Their truth tables are the following:
\begin{center} 
\begin{tabular}{c|cccc}
$\varphi$&  $\lnot \varphi$&$ !\varphi $& $?\varphi$&$\circ\varphi$\\
\hline
$\vt$ & $\vf$ & $\vt$ & $\vt$ & $\vt$ \\
$\vb$  & $\vf$ & $\vt$ & $\vf$ & $\vf$ \\
$\vn$ & $\vt$ & $\vf$ & $\vt$ & $\vf$\\
$\vf$ & $\vt$ & $\vf$ & $\vf$ & $\vt$
\end{tabular}
\end{center}

Notice that $\circ\varphi$ is true iff the truth value of $\varphi$ is $\vt$ or $\vf$, that is, $\circ\varphi$ is true iff $\varphi$ has a classical truth value. So, we will read $\circ\varphi$ as saying that $\varphi$ is \emph{classical}. We see that $\neg\varphi$, ${!}\varphi$, ${?}\varphi$, and $\circ\varphi$ are all classical. Notice that ${!}\varphi$ is true iff $\varphi$ is true, while ${?}\varphi$ is false iff $\varphi$ is false. So, ${!}\varphi$ and ${?}\varphi$ are classical formulas that together completely capture the truth value of $\varphi$.

\smallskip 

Finally, we define the connective $\&$ by 
\begin{itemize}
    \item $\varphi \,\&\, \psi \abbr \no(\varphi\rightarrow \no \psi)$
\end{itemize}

\smallskip

\noindent The following two formulas are valid:
$$\varphi\,\&\,\psi\leftrightarrow \varphi\land\psi \quad \text{ and }\quad\no( \varphi\,\&\,\psi)\leftrightarrow (\varphi\rightarrow\no\psi).$$



\smallskip

It follows that $\varphi\,\&\,\psi$ is true in exactly the same circumstances as $\varphi\land\psi$ is, but it is false when $\varphi$ being true implies that $\psi$ is false.
Also, note that $\varphi\,\&\,\psi$ is classically equivalent to $\land$, so it has just as much of a claim to the title of conjunction. There are many places where, in a classical setting, we would use a conjunction that turns out to be better served by $\&$ than $\land$. As an example, we have
$$\vDash\varphi(f(a))\Leftrightarrow\exists b\left(f(a)=b\,\&\,\varphi(b)\right),$$
but
$$\not\vDash\varphi(f(a))\Leftrightarrow\exists b\left(f(a)=b\land\varphi(b)\right).$$
\subsection{$\BZFC$}
This section reviews the set theory from \cite{KHOMSKII_ODDSSON_2024, HrafnThesis}. The focus is on the aspects most relevant to our development of cardinality while providing enough background for readers unfamiliar with the original material. Not every axiom of $\BZFC$ is covered, and many things are stated without proof. We introduce slight changes in our treatment, most notably in the definition of restricted quantifiers. Additionally, we introduce new notation for characterizing  sets by their
inconsistent, classical, and incomplete parts.

\subsubsection{Extensionality and comprehension}
To start with, we will need to specify when two sets are equal and, equally importantly, when they are unequal. We want two sets $A$ and $B$ to be equal iff the statements ``$x\in A$'' and ``$x\in B$'' get the same truth value for all $x$. We want them to be unequal iff one has an element that the other does not. 
$$\textbf{Extensionality: }\;\;\;\;\; \boldsymbol{ \forall A \forall B ( A=B  \; \Leftrightarrow \;  \forall x(x\in A\Leftrightarrow x\in B))} $$
Now, $A$ and $B$ are equal iff they agree on both their positive extensions (the elements they contain) and their negative extensions (the elements they exclude). We also get
$$A\neq B \leftrightarrow \exists x(x\in A\land x\notin B)\lor \exists x(x\notin A\land x\in B).$$

\noindent Here, and in the remainder of the paper, we use the following abbreviations:
 $$A\neq B \abbr \no(A=B) \quad \text{ and }\quad A\notin B \abbr \no(A\in B).$$

\medskip

We will use the informal notion of a \emph{class} and write $\{x:\varphi(x)\}$ to refer to the class of all objects $y$ that satisfy the property $\varphi$. This motivates the following abbreviation:
$$A=\{x:\varphi(x)\} \abbr \forall x(x\in A\Leftrightarrow \varphi(x)).$$

\medskip 

Two classes of particular importance are 
$$V:=\{x:\top\}\quad \text{ and } \quad \emptyset:=\{x:\bot\}.$$
They are called the \emph{universe} and the \emph{empty set}, respectively. The elements of $V$ are what we call \emph{sets}, and a class is said to be a \emph{proper class} if it is not a set. We have $\forall x(x\in V)$, $\forall x(x\notin \emptyset)$, $\neg \exists x(x\notin V)$, and $\neg \exists x(x\in \emptyset)$. 

It is clear that we cannot add a general comprehension schema, since $\{x:\neg(x\in x)\}$ is provably a proper class. Instead, $\BZFC$ includes the weaker schema:
$$\textbf{Comprehension: }\;\;\;\;\; \boldsymbol{ \forall A \exists B \forall x( x\in B  \; \Leftrightarrow \;  x\in A\land \varphi(x))}$$

\medskip

\noindent This tells us that any subclass of a set is itself a set. So, an appeal to the concept of limitation of size can justify this axiom. Accordingly, we adopt the following abbreviation:
$$\{x\in A:\varphi(x)\}\abbr \{x:x\in A\land \varphi(x)\}.$$

\medskip

We now see that the empty set $\emptyset$ is, in fact, a set, while the universe $V$ is a proper class. 

\subsubsection{Classical and hereditarily classical sets}

We can express that a formula $\varphi$ has a classical truth value using the connective $\circ\varphi$. Extending this idea to sets, we call a set $A$ \emph{classical} when the membership relation ``$ x\in A$'' has a classical truth value for all $x$:

$$\forall x \big({\circ}(x\in A)\big).$$

\smallskip

The nice thing about classical sets is that they behave in a familiar way. For any classical set, an object either belongs to it or does not, never both. However, one should still be careful since even classical sets can have non-classical elements.

We define $\{a_1,...,a_n\}$ as the classical set containing exactly the elements in question. Formally,
$$\{a_1,...,a_n\}:=\{x:{!}(x=a_1)\lor\dots\lor{!}(x=a_n)\}.$$

\smallskip

So even if $a$ happens to be a non-classical set, the \emph{singleton} $\{a\}$ will still be a classical set with exactly one element, namely $a$. 

\medskip

We say that a set is \emph{hereditarily classical} if it is classical, all its elements are classical, the elements of its elements are classical, and so all the way down.\footnote{See \cite[Definition 8.1]{KHOMSKII_ODDSSON_2024} for a proper recursive definition.} We denote the class of all hereditarily classical sets by $\mathbb{HCL}$.

The class $\mathbb{HCL}$ forms an inner model of classical $\ZFC$. This gives us access to the familiar constructions of $\ZFC$ within $\mathbb{HCL}$, including the standard von Neumann ordinals. Consequently, for any classical set $A$ (even one that is not hereditarily classical), we can assign to it the least von Neumann ordinal that stands in one-to-one correspondence with $A$. We denote this ordinal by $|A|_{Cl}$.

\medskip

Beyond classicality, we also have the weaker notions of \emph{consistency} and \emph{completeness}:

\begin{itemize}
\item $A$ is \emph{consistent} if $\forall x\big({!}(x\in A)\rightarrow {?}(x\in A)\big)$ 

\item $A$ is \emph{complete} if $\forall x\big({?}(x\in A)\rightarrow {!}(x\in A)\big)$
\end{itemize}

A set that fails to be consistent is called \emph{inconsistent}, while one that fails to be complete is called \emph{incomplete}. Thus,
\begin{align*}
A \text{ is inconsistent} \quad &\text{ iff }\quad \exists x(x\in A\land x\notin A); \\
A \text{ is incomplete} \quad &\text{ iff }\quad \exists x \neg(x\in A\lor x\notin A).
\end{align*}

\subsubsection{Basic set operations}

We begin with the fundamental notion of one set being contained within another:
$$
A\subseteq B \abbr \forall x(x\in A\Rightarrow x\in B).$$
This gives:
\begin{itemize}
\item $A\not\subseteq B \leftrightarrow \exists x(x\in A\land x\notin B)$ 
      
\item $A=B \Leftrightarrow A\subseteq B \land B\subseteq A$     
\end{itemize}

\medskip
The standard operations of \emph{union}, \emph{intersection}, and \emph{difference} are defined as follows:
\begin{align*}
A\cup B &:=\{x:x\in A\lor x\in B\};  \\
A\cap B &:= \{x:x\in A\land x\in B\};\\
A\setminus B &:= \{x:x\in A\land x\notin B\}.
\end{align*}

\medskip
We make a slight change to the definition of \emph{restricted quantifiers} from \cite{KHOMSKII_ODDSSON_2024,HrafnThesis} to the following:
\begin{align*}
\exists x \in A\,\varphi &\abbr \exists x (x \in A \,\&\,\varphi); \\
\forall x \in A\,\varphi &\abbr \forall x(x\in A\rightarrow \varphi).
\end{align*}
These definitions give the following equivalences:
\begin{itemize}
    \item $\forall x \in A\,\varphi \; \leftrightarrow \; \forall x(x\in A\rightarrow \varphi)$
    \item $\exists x \in A\,\varphi \; \leftrightarrow \; \exists x (x\in A\land \varphi)$
    \item $\no \exists x \in A\,\varphi \; \Leftrightarrow \; \forall x \in A\,\no\varphi$
    \item $\no \forall x \in A\,\varphi \; \Leftrightarrow \; \exists x \in A\,\no\varphi$
\end{itemize}

Notice that our definition of $\exists x \in A\,\varphi$ uses the connective $\&$ instead of $\land$. This is important because if we used the usual $\exists x (x\in A\land \varphi)$, then $\no\exists x\in A \,\varphi$ would be equivalent to $\forall x(x\notin A \lor \no \varphi)$ rather than $\forall x\in A \,\no \varphi$. Moreover, when evaluating the truth value of $\exists x (x \in A \,\&\,\varphi)$, we only need to consider elements of $A$. To see if $\exists x (x \in A \,\&\,\varphi)$ is true, we only need to search for witnesses in $A$, and to see if it is false, we only need to check that $\varphi(x)$ is false for all $x$ from $A$.

\subsubsection{Components of sets}

It can be useful to think of a given set $A$ in terms of its \emph{positive extension}, which consists of those objects making the statement ``$x\in A$'' true, and \emph{negative extension}, which consists of those objects making the statement ``$x\in A$'' false.  These can be formalized as $\{x: {!}(x\in A)\}$ and $\{x:{!}(x\notin A)\}$, respectively. These are both classical, but the problem is that the negative extension turns out to be a proper class. For a dramatic example of this, notice that the negative extension of $\emptyset$ is $V$. 

Therefore, it is more useful to think in terms of the positive extension together with the complement of the negative extension, which is given by $\{x:{?}(x\in A)\}$. For the sake of conceptual symmetry, these are renamed to \emph{!-extension} and \emph{?-extension} of $A$, given by $$A^!:=\{x:{!}(x\in A)\}\quad \text{and}\quad A^?:=\{x:{?}(x\in A)\}, $$
respectively.

Both $A^!$ and $A^?$ can be shown to be sets under the axioms of $\BZFC$. What is important to us is that 
$$x\in A \leftrightarrow x\in A^!\quad \text{ and }\quad x\notin A \leftrightarrow x\notin A^?.$$
We can therefore describe $A$ in terms of the two classical sets $A^!$ and $A^?$.
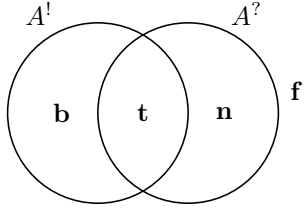
\begin{figure}[H]
\centering
\begin{tikzpicture}[line cap=round,line join=round,>=triangle 45,scale=0.6]
\clip (-4,-2.1) rectangle (4,2.6);
\draw [line width=0.6pt] (-1,0) circle (2);
\draw [line width=0.6pt] (1,0) circle (2);
\node at (-2.3,2.2) {$A^!$};
\node at (2.3,2.2) {$A^?$};
\node at (-1.8,0) {$\vb$};
\node at (0,0) {$\vt$};
\node at (1.8,0) {$\vn$};
\node at (3.4,0.5) {$\vf$};
\end{tikzpicture}
\caption{The four truth values of ``$x\in A$'' with respect to $A^!$ and $A^?$.}
\label{Balls!?}
\end{figure}

We have the following equivalences (\cite[Lemma 4.2]{KHOMSKII_ODDSSON_2024}):

 \begin{itemize}

\item $A \subseteq B \; \leftrightarrow \; A^! \subseteq B^! \land A^? \subseteq B^? $

\item $A \not\subseteq B \; \leftrightarrow \; \exists z (z \in A \land z\notin B) \; \leftrightarrow \;  A^! \not\subseteq B^?$

\item $?(A \subseteq B) \;\leftrightarrow \;   A^! \subseteq B^?$ 

\item $A = B \; \leftrightarrow \;  A^! =  B^! \land A^? = B^?$

\item $A \neq B \; \leftrightarrow \;   \exists z ((z \in A \land z\notin B) \lor (z \in B \land z\notin A)) \; \leftrightarrow \; A^! \not\subseteq B^? \lor B^! \not\subseteq A^?$
 
\item $?(A = B) \; \leftrightarrow \;  A^! \subseteq  B^? \land B^! \subseteq A^?$

\end{itemize}

\smallskip

We can define the \emph{inconsistent}, \emph{classical}, and \emph{incomplete} parts of $A$ by letting
$$A_\vb:=A^!\setminus A^?,\quad A_\vt:=A^!\cap A^?, \quad\text{and}\quad A_\vn:=A^?\setminus A^!,$$
respectively. 
\begin{figure}[H]
\centering
\begin{tikzpicture}[line cap=round,line join=round,>=triangle 45,scale=0.6]
\clip (-4,-2.1) rectangle (4,2.1);
\draw [line width=0.6pt] (-1,0) circle (2);
\draw [line width=0.6pt] (1,0) circle (2);
\node at (-1.6,1) {$A_\vb$};
\node at (0,1) {$A_\vt$};
\node at (1.6,1) {$A_\vn$};
\node at (-1.8,-0.2) {$\vb$};
\node at (0,-0.2) {$\vt$};
\node at (1.8,-0.2) {$\vn$};
\node at (3.4,0.3) {$\vf$};
\end{tikzpicture}
\caption{The four truth values of ``$x\in A$'' with respect to $A_\vt$, $A_\vb$, and $A_\vn$.}
\label{Ballstbf}
\end{figure}

Taking them all together, we get the \emph{realm} of $A$, given by
$$\rlm(A):=A^!\cup A^?\quad (=A_\vb\cup A_\vt\cup A_\vn).$$

\smallskip

\noindent This is the smallest classical set containing $A$, and the structure of $A$ is wholly determined by how the membership relation behaves on $\rlm(A)$. 

We now have \begin{align*}   
\llbracket x\in A \rrbracket^\mathcal{M} &=\begin{cases}
    $\vt$ & \text{if \ $x\in A_\vt$}\\
    $\vb$ & \text{if \ $x\in A_\vb$}\\
    $\vn$ & \text{if \ $x\in A_\vn$}\\
    $\vf$ & \text{if \ $x\notin \rlm (A)$}.
  \end{cases}
\end{align*}

\subsubsection{Non-classical sets}

$\BZFC$ includes a crucial axiom that ensures our theory contains genuinely non-classical sets:

$$\textbf{ACLA:} \quad \exists A (A^! \not\subseteq A^?) \; \land \; \exists A(A^? \not\subseteq A^!)$$

This Anti-Classicality Axiom guarantees the existence of both inconsistent sets (where $A^! \not\subseteq A^?$) and incomplete sets (where $A^? \not\subseteq A^!$). A consequence of ACLA is that for any two classical sets $X$ and $Y$, there exists a unique set $A$ such that $A^!=X$ and $A^?=Y$ (see \cite[Theorem 5.2]{KHOMSKII_ODDSSON_2024}). 

\begin{definition}
    Given a pair of classical sets $X$ and $Y$, we denote the unique set $A$ such that $A^!=X$ and $A^?=Y$ by $\langle X|Y\rangle$.
\end{definition}

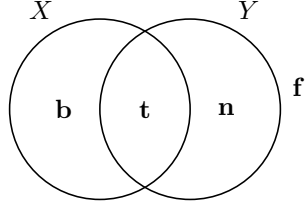
\begin{figure}[H]
\centering
\begin{tikzpicture}[line cap=round,line join=round,>=triangle 45,scale=0.6]
\clip (-4,-2.1) rectangle (4,2.6);
\draw [line width=0.6pt] (-1,0) circle (2);
\draw [line width=0.6pt] (1,0) circle (2);
\node at (-2.3,2.2) {$X$};
\node at (2.3,2.2) {$Y$};
\node at (-1.8,0) {$\vb$};
\node at (0,0) {$\vt$};
\node at (1.8,0) {$\vn$};
\node at (3.4,0.5) {$\vf$};
\end{tikzpicture}
\caption{The four truth values of ``$x\in \langle X|Y\rangle$'' with respect to $X$ and $Y$.}
\label{XY}
\end{figure}



Another consequence of the $\mathrm{ACLA}$ axiom is that given three disjoint classical sets $X$, $Y$, and $Z$, there is a unique set $A$ such that $A_\vb=X$, $A_\vt=Y$, and $A_\vn=Z$. 

\begin{definition}\label{def:XYZ}
For pairwise disjoint classical sets $X$, $Y$, and $Z$, we write $\langle X|Y|Z\rangle$ to denote the unique set $A$ with $A_\vb = X$, $A_\vt = Y$, and $A_\vn = Z$.  
\end{definition}

\begin{figure}[H]
\centering
\begin{tikzpicture}[line cap=round,line join=round,>=triangle 45,scale=0.6]
\clip (-4,-2.1) rectangle (4,2.1);
\draw [line width=0.6pt] (-1,0) circle (2);
\draw [line width=0.6pt] (1,0) circle (2);
\node at (-1.6,1) {$X$};
\node at (0,1) {$Y$};
\node at (1.6,1) {$Z$};
\node at (-1.8,-0.2) {$\vb$};
\node at (0,-0.2) {$\vt$};
\node at (1.8,-0.2) {$\vn$};
\node at (3.4,0.3) {$\vf$};
\end{tikzpicture}
\caption{The four truth values of ``$x\in \langle X|Y|Z\rangle$'' with respect to $X$, $Y$, and $Z$.}
\label{XYZ}
\end{figure}
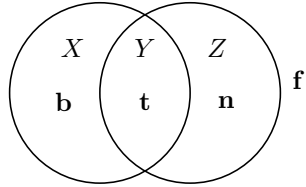

The two bracket notations are connected by the identities:
\begin{align*}
\langle X|Y\rangle &= \langle X\setminus Y|X\cap Y|Y\setminus X\rangle; \\
\langle X|Y|Z\rangle &= \langle X\cup Y|Y\cup Z\rangle.
\end{align*}
In the second equality, $X$, $Y$, and $Z$ are pairwise disjoint. Finally,
$$A=\langle A^!|A^?\rangle=\langle A_\vb| A_\vt|A_\vn\rangle.$$

Both types of notation have their own advantages, but for our purposes, we will tend to rely on the three-part notation $\langle X|Y|Z\rangle$. This is because it expresses the structure of a set in a more direct way: When using the two-part notation, the truth value of ``$a\in \langle \{a,b,c\}|\{c,d\}\rangle$'' depends on whether $a\in \{c,d\}$. On the other hand, we can directly see that the truth value of ``$a\in \langle \{a,b\}|\{c\}|\{d\}\rangle$'' is $\vb$. 
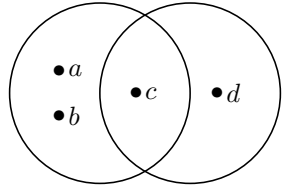
\begin{figure}[H]
\centering
\begin{tikzpicture}[line cap=round,line join=round,>=triangle 45,scale=0.6]
\clip (-3.8,-2.2) rectangle (3.8,2.2);
\draw [line width=0.6pt] (-1,0) circle (2);
\draw [line width=0.6pt] ( 1,0) circle (2);
\node at (-1.9, 0.5) {$\bullet$}; \node at (-1.9, 0.5) [anchor=west] {$a$};
\node at (-1.9,-0.5) {$\bullet$}; \node at (-1.9,-0.5) [anchor=west] {$b$};
\node at (-0.2, 0)   {$\bullet$}; \node at (-0.2, 0)   [anchor=west] {$c$};
\node at ( 1.6, 0)   {$\bullet$}; \node at ( 1.6, 0)   [anchor=west] {$d$};
\end{tikzpicture}
\caption{The set $\langle \{a,b\}|\{c\}|\{d\}\rangle$.}
\label{fig:bracket-notation}
\end{figure}

\subsubsection{Functions}
We will assume a suitable encoding\footnote{The specific encoding can be found in \cite[p. 981]{KHOMSKII_ODDSSON_2024}.} of the \emph{ordered pair} $(a,b)$ that satisfies
$$(a,b)=(c,d)\Leftrightarrow (a=c \land b=d).$$
By a \emph{classical function}, we mean a classical set $f$ of ordered pairs such that
$$\left((a,b)\in f\land (a,c)\in f\right)\rightarrow {!}(b=c).$$
Thus, a classical function is just a normal function we are familiar with from classical set theory that takes an input $a$ and spits out an output $f(a)$. For our purposes, classical functions will suffice.

\begin{remark}
    The reason we include the $!$-sign in the definition is to make sure that the statement ``$f$ is a classical function'' is itself classical. The results of this paper do not hinge on this fact, and if the reader prefers, they are welcome to follow along with a $!$-free definition. 
\end{remark}
The \emph{domain} of $f$ is given by $$\mathrm{dom}(f):=\{x:\exists y((x,y)\in f)\},$$ which is a classical set since $f$ is classical. If $A$ is any set such that $A\subseteq \mathrm{dom}(f)$, then the \emph{image} of $A$ under $f$ is given by
$$f[A]:=\{f(x):x\in A\}:=\{y:\exists x(y=f(x)\,\&\,x\in A)\}.$$
Notice that $b\in f[A]$ is true when there is an $a\in A$ with $f(a)=b$. It is false when $a\notin A$ for every $a$ with $f(a)=b$. This tells us that 
$$f[A^!]=(f[A])^!\quad \text{ and }\quad f[A^?]=(f[A])^?.$$
A good way to think of this is that $\llbracket b\in f[A]\rrbracket$ is the join of $\llbracket a\in A\rrbracket$ taken over all $a$ that get mapped to $b$.
\begin{example}
    Take $A=\langle\{1,2,3\}|\{3,4,5\}\rangle$ and let $f$ be the function from $\rlm(A)$ given by \begin{center} 
\begin{tabular}{c|ccccc}
$x$ &  $1$ & $2$ & $3$ & $4$ & $5$ \\
\hline
$f(x)$ & $a$ & $b$ & $c$ & $d$ & $b$ \\
\end{tabular}
\end{center}
Then $f[A]=\langle \{a,b,c\}|\{b,c,d\}\rangle.$
\begin{figure}[H]
    \centering
\begin{tikzpicture}[scale=0.8, 
  >={Stealth[length=5pt]}, thick,
  shorten >=2pt, shorten <=2pt,
  every node/.style={inner sep=1pt, font=\small}
]
\draw (0, 1) ellipse (1.25 and 1.7);
\draw (0,-1) ellipse (1.25 and 1.7);
\node at (-1.65, 2) {$A^!$};
\node at (-1.65,-2) {$A^?$};
\node[label=left:$1$] (a1) at (0, 2) {$\bullet$};
\node[label=left:$2$] (a2) at (0, 1.1) {$\bullet$};
\node[label=left:$3$] (a3) at (0, 0.0) {$\bullet$};
\node[label=left:$4$] (a5) at (0,-1.1) {$\bullet$};
\node[label=left:$5$] (a4) at (0,-2) {$\bullet$};
 
\begin{scope}[xshift=6cm]
\draw (0, 1) ellipse (1.25 and 1.7);
\draw (0,-1) ellipse (1.25 and 1.7);
\node at (1.9, 2) {$(f[A])^!$};
\node at (1.9,-2) {$(f[A])^?$};
\node[label=right:$a$] (fa) at (0, 1.5) {$\bullet$};
\node[label=right:$b$] (fb) at (0, 0.3) {$\bullet$};
\node[label=right:$c$] (fc) at (0,-0.3) {$\bullet$};
\node[label=right:$d$] (fd) at (0,-1.4) {$\bullet$};
\end{scope}
 
\draw[->] (a1) -- (fa);
\draw[->] (a2) -- (fb);
\draw[->] (a3) -- (fc);
\draw[->] (a4) -- (fb);
\draw[->] (a5) -- (fd);
\end{tikzpicture}
    \nocaption
    \label{fig:Fun}
\end{figure}

\end{example}

We say that $f$ is an \emph{injection} if 
$$\left((a,c)\in f\land (b,c)\in f\right)\rightarrow {!}(a=b).$$
In which case, we get
$$f[A_\vb]=(f[A])_\vb, \quad f[A_\vt]=(f[A])_\vt, \quad \text{and}\quad f[A_\vn]=(f[A])_\vn.$$

\begin{example}
    Take $A=\langle\{1,2\}|\{3\}|\{4\}\rangle$ and let $f$ be the following injection from $\rlm(A)$:
\begin{center} 
\begin{tabular}{c|ccccc}
$x$ &  $1$ & $2$ & $3$ & $4$ \\
\hline
$f(x)$ & $a$ & $b$ & $c$ & $d$ \\
\end{tabular}
\end{center}
Then $f[A]=\langle \{a,b\}|\{c\}|\{d\}\rangle.$
The important thing to note is that $f[A]$ has precisely the same structure as $A$.

\begin{figure}[H]
    \centering
\begin{tikzpicture}[scale=0.7,
  >={Stealth[length=5pt]}, thick,
  shorten >=2pt, shorten <=2pt,
  every node/.style={inner sep=1pt, font=\small}
]
\draw (0, 1) ellipse (1.25 and 1.7);
\draw (0,-1) ellipse (1.25 and 1.7);
\node at (-1.65, 2) {$A^!$};
\node at (-1.65,-2) {$A^?$};
\node[label=left:$1$] (a1) at (0, 2) {$\bullet$};
\node[label=left:$2$] (a2) at (0, 1.2) {$\bullet$};
\node[label=left:$3$] (a3) at (0, 0.0) {$\bullet$};
\node[label=left:$4$] (a4) at (0,-1.5) {$\bullet$};
 
\begin{scope}[xshift=6cm]
\draw (0, 1) ellipse (1.25 and 1.7);
\draw (0,-1) ellipse (1.25 and 1.7);
\node at (1.9, 2) {$(f[A])^!$};
\node at (1.9,-2) {$(f[A])^?$};
\node[label=right:$a$] (fa) at (0, 2) {$\bullet$};
\node[label=right:$b$] (fb) at (0, 1.2) {$\bullet$};
\node[label=right:$c$] (fc) at (0, 0.0) {$\bullet$};
\node[label=right:$d$] (fd) at (0,-1.5) {$\bullet$};
\end{scope}
 
\draw[->] (a1) -- (fa);
\draw[->] (a2) -- (fb);
\draw[->] (a3) -- (fc);
\draw[->] (a4) -- (fd);
\end{tikzpicture}
    \nocaption
    \label{fig:Injection}
\end{figure}

\end{example}

\smallskip

Lastly, we will be using the following form of the axiom of choice:
$$\textbf{Choice: \; Any classical set of non-empty sets has a choice function.}$$

\noindent Here, a choice function is simply a function such that $f(x)\in x$ for each $x\in \mathrm{dom}(f)$.

\subsubsection{Products and disjoint unions}

The \emph{Cartesian product} of sets $A$ and $B$ is given by
$$A\times B:=\{(x,y):x\in A\land y\in B\}.$$
This gives us the expected membership condition for the Cartesian product:
$$(a,b)\in A\times B \Leftrightarrow a\in A \land b\in B.$$
Similarly, the \emph{disjoint union} of $A$ and $B$ is given by
$$A \uplus B:=\{(x,0):x\in A\}\cup \{(x,1):x\in B\}.$$
We have
$$(a,0)\in A\uplus B \Leftrightarrow a\in A;$$
$$(b,1)\in A\uplus B \Leftrightarrow b\in B.$$

\section{An informal treatment of cardinality}\label{Sec:An Informal Treatment of Cardinality}
When dealing with familiar mathematical concepts in a non-classical setting, it can be tempting to simply copy old definitions and continue from there. However, we must be careful, as our intuitions about these concepts are typically formed against a classical background, and the familiar definitions are motivated for a classical theory. We must therefore dedicate some space to developing our intuition for these concepts in our new setting. In this section, we develop an intuitive picture of cardinality in our paraconsistent and paracomplete setting. This will then serve as our guide in later sections, where we will formalize this notion in $\BZFC$. 

\subsection{Cardinal numbers by abstraction}
As a starting point, we take the following passage from Cantor:
\begin{quote}
     We call by the name ``power" or ``cardinal number" of $M$ the general concept which, by means of our active faculty, arises from the aggregate $M$ when we make abstraction of the nature of its various elements $m$ and of the order in which they are given. \cite[p.~85]{Cantor1955}
\end{quote}
He calls what remains of an element after the above-mentioned abstraction a \emph{unit}. Thus, a cardinal number to Cantor is a set of such units. Although this is not exactly precise, we can use this as a starting point. 

\smallskip

While Cantor only talks about the elements of a set, we need to consider what happens in the whole realm of a set. So, we will informally think of the cardinality of a set as what remains when we abstract away all the particulars about the elements of the realm of said set. 

For example, the cardinal number of the set $A=\langle \{a,b\}| \{c\}| \{d\}\rangle$ is something like
$$|A|=\langle \{\bullet,\bullet\}| \{\bullet\}| \{\bullet\}\rangle,$$
Where each instance of $\bullet$ represents a unit. Crucially, this object captures only the structure, or the size, of $A$ and nothing else.

\begin{figure}[H]
\centering
\begin{tikzpicture}[line cap=round,line join=round,>=triangle 45,scale=0.6]
\clip (-3.8,-2.2) rectangle (3.8,2.2);
\draw [line width=0.6pt] (-1,0) circle (2);
\draw [line width=0.6pt] ( 1,0) circle (2);
\node at (-1.9, 0.5) {$\bullet$}; \node at (-1.9, 0.5) [anchor=west] {$a$};
\node at (-1.9,-0.5) {$\bullet$}; \node at (-1.9,-0.5) [anchor=west] {$b$};
\node at (-0.2, 0)   {$\bullet$}; \node at (-0.2, 0)   [anchor=west] {$c$};
\node at ( 1.6, 0)   {$\bullet$}; \node at ( 1.6, 0)   [anchor=west] {$d$};
\end{tikzpicture}
\begin{tikzpicture}[line cap=round,line join=round,>=triangle 45,scale=0.6]
\clip (-3.8,-2.3) rectangle (3.8,2.3);
\draw [line width=0.6pt] (-1,0) circle (2);
\draw [line width=0.6pt] (1,0) circle (2);
\node at (-1.75,0.5) {$\bullet$};
\node at (-1.75,-0.5) {$\bullet$};
\node at (0,0)   {$\bullet$};
\node at (1.75,0)   {$\bullet$};
\end{tikzpicture}
\caption{The set $\langle \{a,b\}|\{c\}|\{d\}\rangle$ and its cardinal number.}
\label{fig:cardexampleabc}
\end{figure}
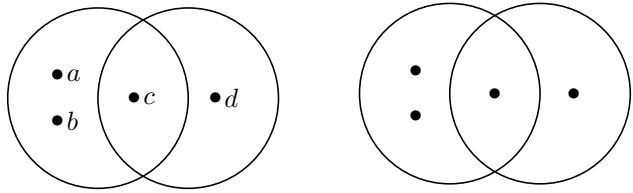

We will use the following criterion for equality between cardinal numbers: Two cardinal numbers are to be considered equal if they are both the cardinality of the same set. Similarly, we think of them as unequal if they are only the cardinalities of sets that are unequal:
\begin{align*}
\kappa= \mu \quad&\text{ iff }\quad\text{there is a set with cardinality both $\kappa$ and $\mu$;} \\[0.5em]
\kappa \neq \mu \quad&\text{ iff }\quad A\neq B\text{ for any pair of sets $A$ and $B$ such that  } \\
&\phantom{\text{ iff }}\quad\text{$\kappa$ is the cardinality of $A$ and $\mu$ is the cardinality of $B$.}
\end{align*}

Two cardinal numbers are therefore equal if they agree on how many units are in the classical, inconsistent, and incomplete parts. They are unequal if their structure alone mandates that one has a unit that the other does not. As an example, $$\langle \{\bullet,\bullet\}| \{\bullet\}|\{\bullet\}\rangle\neq \langle \{\bullet\}| \{\bullet\}|\{\bullet\}\rangle,$$
\smallskip
since the $!$-extension of the former has three units while the $?$-extension of the latter only has two. So there will be a unit that belongs to the former but not the latter.
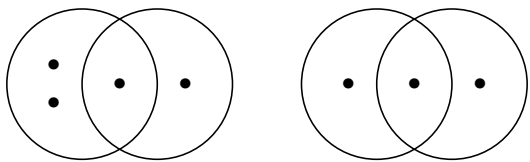
\begin{figure}[H]
\centering

\begin{tikzpicture}[line cap=round,line join=round,scale=0.5]
\clip (-3.8,-2.3) rectangle (3.8,2.3);
\draw [line width=0.6pt] (-1,0) circle (2);
\draw [line width=0.6pt] ( 1,0) circle (2);
\node at (-1.75, 0.5) {$\bullet$};
\node at (-1.75,-0.5) {$\bullet$};
\node at ( 0,    0)   {$\bullet$};
\node at ( 1.75, 0)   {$\bullet$};
\end{tikzpicture}
\begin{tikzpicture}[line cap=round,line join=round,scale=0.5]
\clip (-3.8,-2.3) rectangle (3.8,2.3);
\draw [line width=0.6pt] (-1,0) circle (2);
\draw [line width=0.6pt] ( 1,0) circle (2);
\node at (-1.75, 0) {$\bullet$};
\node at ( 0,    0) {$\bullet$};
\node at ( 1.75, 0) {$\bullet$};
\end{tikzpicture}
\caption{A pair of cardinals that are unequal to each other.}
\label{fig:twocard}
\end{figure}

The cardinals are then ordered as follows:
\begin{align*}
\kappa\leq \mu \quad&\text{ iff }\quad A\subseteq B\text{ for some pair of sets $A$ and $B$ such that  } \\
&\phantom{\text{ iff }}\quad\text{$\kappa$ is the cardinality of $A$ and $\mu$ is the cardinality of $B$;}\\[0.5em]
\kappa \nleq \mu \quad&\text{ iff }\quad A\nsubseteq B\text{ for any pair of sets $A$ and $B$ such that  } \\
&\phantom{\text{ iff }}\quad\text{$\kappa$ is the cardinality of $A$ and $\mu$ is the cardinality of $B$.}
\end{align*}

\subsection{Counting, tallies, and equinumerosity}
Here we will take one more step towards a formal account of cardinality. We start by dispensing with the notions of units and abstraction in favor of a story of \emph{counting} sets using \emph{pebbles}. The only thing we assume about these pebbles is that we have an inexhaustible supply of them.

We count the elements of a set $A$ by assigning each element in the realm of $A$ a pebble in an injective manner. We call the resulting set of pebbles a \emph{tally} of $A$, and we will call any set of pebbles a \emph{tally}. We can further relate tallies in such a way that we can think of the cardinal numbers as equivalence classes\footnote{See \cite[Definition 6.1]{KHOMSKII_ODDSSON_2024} for equivalence relations in $\BZFC$.} of tallies:
\begin{align*}
\mathcal{C}\approx \mathcal{D} \quad&\text{ iff }\quad\text{$\mathcal{C}$ and $\mathcal{D}$ are both tallies of the same set;} \\[0.5em]
\mathcal{C} \not\approx \mathcal{D} \quad&\text{ iff }\quad A\neq B\text{ for any pair of sets $A$ and $B$ such that  } \\
&\phantom{\text{ iff }}\quad\text{$\mathcal{C}$ is a tally of $A$ and $\mathcal{D}$ is a tally of $B$.}
\end{align*}

Similarly, we can order tallies as follows:
\begin{align*}
\mathcal{C}\leq \mathcal{D} \quad&\text{ iff }\quad A\subseteq B\text{ for some pair of sets $A$ and $B$ such that  } \\
&\phantom{\text{ iff }}\quad\text{$\mathcal{C}$ is a tally of $A$ and $\mathcal{D}$ is a tally of $B$;}\\[0.5em]
\mathcal{C} \nleq \mathcal{D} \quad&\text{ iff }\quad A\nsubseteq B\text{ for any pair of sets $A$ and $B$ such that  } \\
&\phantom{\text{ iff }}\quad\text{$\mathcal{C}$ is a tally of $A$ and $\mathcal{D}$ is a tally of $B$.}
\end{align*}

This now leads us to the following way to use tallies to compare the sizes of sets: Two sets have the same size if we can count them using the same tally. And their sizes are different if we have to use different tallies to count them:
\begin{align*}
A\cong B \quad&\text{ iff }\quad\text{there is a set of pebbles that is a tally for both $A$ and $B$;} \\[0.5em]
A\not\cong B \quad&\text{ iff }\quad\text{for any pair of tallies for $A$ and $B$, one will have} \\
&\phantom{\text{ iff }}\quad\text{a pebble that the other does not.}\footnotemark
\end{align*}\footnotetext{When we say that a set $A$ (or a tally) does not have an element $a$, we mean that $a\notin A.$}
This tells us that $A\cong B$ is true precisely when $A$ and $B$ have the same structure, meaning that the classical, inconsistent, and incomplete parts of $A$ have the same cardinality as the corresponding parts of
$B$ (in the classical sense).

The interesting part is the clause for when $A\cong B$ is false. This happens when, no matter how we tally the two sets, there will always be a pebble in one tally that is not in the other. 
\begin{example}
    Consider the set $A=\langle \{a, b\}|\{c\}|\{d\}\rangle$. Clearly, $A\cong A$, as is the case with any set. We also find that $A\not\cong A$. To see why, take any pair of tallies $\mathcal{C}=\langle\{p_1, p_2\}|\{p_3\}|\{p_4\}\rangle$ and $\mathcal{D}=\langle\{q_1,q_2\}|\{q_3\}|\{q_4\}\rangle$ of $A$. Now, $\mathcal{C}^!=\{p_1,p_2,p_3\}$ has three elements, while $\mathcal{D}^?=\{q_3,q_4\}$ only has two. It follows that there is a pebble in $\mathcal{C}^!$ that is not in $\mathcal{D}^?$. Said pebble will therefore be an element of $\mathcal{C}$, but a non-element of $\mathcal{D}$.
\end{example}

We can also use pebbles to say that the size of $A$ is at most that of $B$: 
\begin{align*}
A\preceq B \quad&\text{ iff }\quad\text{there is a tally of $A$ that is a subset of a tally of $B$;} \\[0.5em]
A\not\preceq B \quad&\text{ iff }\quad\text{for any pair of tallies for $A$ and $B$, there will be a pebble} \\
&\phantom{\text{ iff }}\quad\text{that is in the tally of $A$ but not in the tally of $B$.}
\end{align*}

Thus, $A\preceq B$ tells us that we can fit a tally of $A$ within a tally of $B$, while $A\not\preceq B$ tells us that no tally of $A$ will fit within a tally of $B$, in the sense that there will always be a pebble left over.

\begin{proposition}
    If $A$ and $B$ are sets with tallies $\mathcal{C}$ and $\mathcal{D}$, then
    $$A\cong B\Leftrightarrow \mathcal{C}\approx \mathcal{D}\quad\text{ and }\quad A\preceq B\Leftrightarrow \mathcal{C}\leq \mathcal{D}.$$
\end{proposition}
\begin{proof}
    We only show that $A\not\cong B\leftrightarrow \mathcal{C}\not\approx\mathcal{D}$. First, assume that $\mathcal{C}\not\approx\mathcal{D}$, and let $\mathcal{C}'$ and $\mathcal{D}'$ be tallies of $A$ and $B$, respectively. Then $\mathcal{C}'$ and $\mathcal{D}'$ are themselves sets with tallies $\mathcal{C}$ and $\mathcal{D}$, so $\mathcal{C}'\neq\mathcal{D}'$. Thus, $A\not\cong B$.

    Now assume that $A\not\cong B$, and let $A'$ and $B'$ be sets with tallies $\mathcal{C}$ and $\mathcal{D}$, respectively. Let $f$ be any injection that replaces each element in $\rlm(A')\cup\rlm(B')$ with a pebble. Then, $f[A']$ and $f[B']$ are also tallies of $A$ and $B$, so $f[A']\neq f[B']$. Since $f$ is an injection, we also get $A'\neq B'$.
    
\end{proof}

\section{Equinumerosity in $\BZFC$}\label{EquinumerositySection}
We will now formalize our notion of equinumerosity in $\BZFC$ and establish some of its consequences.

\smallskip

In the previous section, we used the informal notion of sets of pebbles to compare the sizes of sets. The only assumption we made about the pebbles is that there are enough of them. Since the only objects of $\BZFC$ are sets, we can simply take the convention that any set is a pebble.

With this in mind, we call any injection from $\rlm (A)$ a \emph{counting} of $A$. We denote the class of all countings of $A$ by $\mathrm{Count}(A)$. In this way, we think of a tally of a set $A$ as the result of counting it. That is, if $f$ is a counting of $A$, we think of $f[A]$ as the resulting tally. This leads us to the following definition.

\begin{definition} Given two sets $A$ and $B$, we write $A\cong B$, and say that $A$ and $B$ are \emph{equinumerous}, if $$\exists f\in \mathrm{Count}(A)\exists g\in \mathrm{Count}(B)(f[A]=g[B]).$$ Similarly, we write $A \preceq B$, and say that the \emph{cardinality of} $A$ \emph{is at most that of} $B$, if $$\exists f\in \mathrm{Count(A)}\exists g\in \mathrm{Count}(B)(f[A]\subseteq g[B]).$$
\end{definition}

\begin{remark}
Keeping in mind the falsity conditions for the restricted quantifiers, we see that 
\begin{align*}
A\not\cong B \quad&\text{ iff }\quad\forall f\in \mathrm{Count(A)}\forall g\in \mathrm{Count}(B)(f[A]\neq g[B]); \\[0.5em]
A\not\preceq B \quad&\text{ iff }\quad\forall f\in \mathrm{Count(A)}\forall g\in \mathrm{Count}(B)(f[A]\not\subseteq g[B]).
\end{align*}
Our formal definition, therefore, matches our informal notion from the previous section. 
\end{remark}

\begin{remark}
If $A$ and $B$ are classical sets, then the statements $A\cong B$ and $A \preceq B$ match those from classical set theory. We can therefore help ourselves to the usual results about them. 
\end{remark}

It turns out that in order to evaluate $A\cong B$, it suffices to quantify over the countings of $A$. This way, $A\cong B$ can be understood as saying that $B$ is equal to a tally of $A$.

\begin{proposition}
    For all $A$ and $B$,
    \begin{align*}
        A\cong B&\Leftrightarrow \exists f\in \mathrm{Count(A)}(f[A]=B);\\[0.5em]
        A\preceq B&\Leftrightarrow \exists f\in \mathrm{Count(A)}(f[A]\subseteq B).
    \end{align*}
\end{proposition}

\begin{proof}
    We will only show that $\no\exists f\in \mathrm{Count(A)}(f[A]=B)$ implies $A\not\cong B$. The others are straightforward.

    Since $\rightarrow$ contraposes with respect to $\neg$, it suffices to show that $${?}(A\cong B)\rightarrow {?}(\exists f\in \mathrm{Count(A)}(f[A]=B)),$$
    which is equivalent to 
$$
\begin{aligned}
&\exists f\in \mathrm{Count}(A)\exists g\in \mathrm{Count}(B){?}(f[A]=g[B])\\
&\rightarrow\exists f\in \mathrm{Count}(A){?}(f[A]=B).
\end{aligned}
$$

We fix $f\in\mathrm{Count}(A)$ and $g\in \mathrm{Count}(B)$ such that ${?}(f[A]=g[B])$. Our aim is to find $h\in \mathrm{Count}(A)$ such that ${?}(h[A]=B)$.

\smallskip

Since ${?}(f[A]=g[B])$, we have that $f[A]^!\subseteq g[B]^?$ and $g[B]^!\subseteq f[A]^?$. It also follows that $f[A]_\vb\subseteq g[B]_\vn$ and $g[B]_\vb\subseteq f[A]_\vn$. These are all consequences of the axiom of extensionality.

We pick any injection $l\in \mathrm{Count}(A)$ such that $l[\rlm (A)]\cap \rlm(B)=\emptyset$ and define $h$ as follows:
\begin{align*}   
h(x)&:=\begin{cases}
    g^{-1}(f(x)) & \text{if \ $f(x)\in g[\rlm(B)]$}\\
    l(x) & \text{else}
  \end{cases}
\end{align*}

All that remains is to show that ${?}(h[A]=B)$, or equivalently,  $h[A]^!\subseteq B^?$ and $B^!\subseteq h[A]^?$:  If $x\in A^!$, then $f(x)\in f[A]^!\subseteq g[B]^?.$ So, $h(x)=g^{-1}(f(x))\in B^?$. 

If $y\in B^!$, then $g(y)\in g[B]^!\subseteq f[A]^?$. So, $y=g^{-1}(f(x))$ for some $x\in A^?$. Thus, $y\in h[A]^?$.

\end{proof}

The next theorem tells us that equinumerosity is an equivalence relation in a strong sense.\footnote{Again, see \cite[Definition 6.1]{KHOMSKII_ODDSSON_2024} for an account of equivalence relations in $\BZFC$.} 

\begin{theorem}\label{long}
For all $A$, $B$ and $C$,
\begin{enumerate}
\item $A\cong A$;
\item $A\cong B\Leftrightarrow B\cong A$;
\item if $A\cong B$, then $A\cong C\Leftrightarrow B\cong C$;
\item if $A\cong B$, then $A\preceq C\Leftrightarrow B \preceq C$;
\item if $A\cong B$, then $C \preceq A\Leftrightarrow C \preceq B$;
\item if $A \preceq B$ and $B\preceq C$, then $A\preceq C$.
\end{enumerate}
\end{theorem}
\begin{proof}
We will only show that $A\cong B\rightarrow (A\not\cong C\rightarrow B\not\cong C)$:

\smallskip

Assume that $A\cong B$ and $A\ncong C$, and fix an arbitrary $g\in \mathrm{Count}(B)$. Since $A\cong B$, we know that there is an $f\in \mathrm{Count}(A)$ such that $f[A]=B$. We also know that $g[f[A]]\neq C$, since $g\circ f$ is a counting of $A$. Thus, $g[B]\neq C.$

\end{proof}

Somewhat surprisingly, the Schröder–Bernstein theorem fails for sets with infinite realms. We will later see that it holds for sets with finite realms (Theorem~\ref{FinShrBern}).

\begin{proposition}\label{FailureOfSchrBern}
    There are sets $A$ and $B$ such that $A\preceq B$ and $B\preceq A$, but $\neg(A\cong B)$.
\end{proposition}
\begin{proof}
    We take $A:=\langle \{0\}\mid\omega\setminus\{0\}\mid\emptyset\rangle$, $B=\omega$, and $C:=\langle \{\omega\}\mid\omega\mid\emptyset\rangle$ (See Figure \ref{figomegaAB}). We have $A\subseteq B$, $B\subseteq C$, and $A\cong C$. Thus, $A\preceq B$ and $B\preceq A$. But clearly,  $\neg (A\cong B).$

    \smallskip
    
\begin{figure}[H]
\centering
\begin{tikzpicture}[line cap=round,line join=round,scale=0.5]
\clip (-3.8,-2.3) rectangle (3.8,2.4);
\draw [line width=0.6pt] (-1,0) circle (2);
\draw [line width=0.6pt] ( 1,0) circle (2);
\node at (-2.9, 2.0) {$A^!$};
\node at ( 2.9, 2.0) {$A^?$};
\node[font=\footnotesize] at (-2.05, 0) {$\bullet$}; \node[font=\footnotesize] at (-2.1, 0) [anchor=west] {$0$};
\node[font=\tiny] at (-0.1, 0.8) {$\bullet$}; \node[font=\tiny] at (-0.1, 0.8) [anchor=west] {$1$};
\node[font=\tiny] at (-0.1, 0.4) {$\bullet$}; \node[font=\tiny] at (-0.1, 0.4) [anchor=west] {$2$};
\node[font=\tiny] at (-0.1,0) {$\bullet$}; \node[font=\tiny] at (-0.1,0) [anchor=west] {$3$};
\node at (0.1,-0.45) {$\vdots$};
\end{tikzpicture}
\begin{tikzpicture}[line cap=round,line join=round,scale=0.5]
\clip (-3.8,-2.3) rectangle (3.8,2.4);
\draw [line width=0.6pt] (-1,0) circle (2);
\draw [line width=0.6pt] ( 1,0) circle (2);
\node at (-2.9, 2.0) {$B^!$};
\node at ( 2.9, 2.0) {$B^?$};
\node[font=\tiny] at (-0.1, 0.8) {$\bullet$}; \node[font=\tiny] at (-0.1, 0.8) [anchor=west] {$0$};
\node[font=\tiny] at (-0.1, 0.4) {$\bullet$}; \node[font=\tiny] at (-0.1, 0.4) [anchor=west] {$1$};
\node[font=\tiny] at (-0.1,0) {$\bullet$}; \node[font=\tiny] at (-0.1,0) [anchor=west] {$2$};
\node at (0.1,-0.45) {$\vdots$};
\end{tikzpicture}    \begin{tikzpicture}[line cap=round,line join=round,scale=0.5]
\clip (-3.8,-2.3) rectangle (3.8,2.4);
\draw [line width=0.6pt] (-1,0) circle (2);
\draw [line width=0.6pt] ( 1,0) circle (2);
\node at (-2.9, 2.0) {$C^!$};
\node at ( 2.9, 2.0) {$C^?$};
\node[font=\footnotesize] at (-2.05, 0) {$\bullet$}; \node[font=\footnotesize] at (-2.1, 0) [anchor=west] {$\omega$};
\node[font=\tiny] at (-0.1, 0.8) {$\bullet$}; \node[font=\tiny] at (-0.1, 0.8) [anchor=west] {$0$};
\node[font=\tiny] at (-0.1, 0.4) {$\bullet$}; \node[font=\tiny] at (-0.1, 0.4) [anchor=west] {$1$};
\node[font=\tiny] at (-0.1,0) {$\bullet$}; \node[font=\tiny] at (-0.1,0) [anchor=west] {$2$};
\node at (0.1,-0.45) {$\vdots$};
\end{tikzpicture}
\caption{}
\label{figomegaAB}
\end{figure}
\end{proof}

The next theorem tells us  that we can completely characterize our notion of equinumerosity by only comparing the sizes of classical sets.

\begin{theorem}\label{short}
For all $A$ and $B$,

\begin{enumerate} 
    \item $A \preceq B\quad\text{ iff }\quad \rlm(A)\preceq \rlm(B)\text{ and } A^!\preceq B^!\text{ and }A^?\preceq B^?\text{ and }A_\vt \preceq B_\vt$;
    \item $A \not\preceq B\quad\text{ iff }\quad A^!\not\preceq B^?$;
    \vspace{0.5em}
    \item $A \cong B\quad\text{ iff }\quad A_\vt\cong B_\vt\text{ and } A_\vb\cong B_\vb\text{ and }A_\vn\cong B_\vn$;
    \item $A\ncong B\quad\text{ iff }\quad A^!\npreceq B^?\text{ or } B^!\npreceq A^?\text{ or } A_\vb \npreceq B_\vn\text{ or }B_\vb\npreceq A_\vn$.
\end{enumerate}
\end{theorem}

\begin{proof}\

\medskip

\noindent\textbf{1)} We only show the direction from right to left, as the other is straightforward. Assume that $\rlm(A)\preceq \rlm(B)$, $A^!\preceq B^!$, $A^?\preceq B^?$, and $A_\vt \preceq B_\vt$. We now construct the required injection $f$ by considering cases based on whether the sets $B_\vt$, $B_\vb$, and $B_\vn$ are finite.

\smallskip

Case 1: If $B_\vt$ is infinite, then $B\cong C,$ where 
$$C:= \big\langle B_\vb\times \{0\}\mid (B_\vt\times \{1,2,3\})\mid B_\vn\times \{4\}\big\rangle.$$
We will show that $A\preceq C.$
Fix injections $g:A^!\to B^!$, $h:A^?\to B^?$, and $j:A_\vt \to B_\vt$.\footnote{In this case we don't need to use that $\rlm(A)\preceq \rlm(B)$.} We let 
$$ f(x) = 
  \begin{cases}
    (g(x),0), & \text{if } x\in A_\vb\text{ and }g(x)\in B_\vb \\
    (g(x),1), & \text{if } x\in A_\vb\text{ and }g(x)\in B_\vt \\
    (j(x),2), & \text{if } x\in A_\vt \\
    (h(x),3), & \text{if } x\in A_\vn\text{ and }h(x)\in B_\vt \\
    (h(x),4), & \text{if } x\in A_\vn\text{ and }h(x)\in B_\vn. \\
  \end{cases}
$$ 
Clearly, $f$ is an injection defined on $\rlm(A)$. To see that $f[A]\subseteq C$, one checks by cases that if $x\in A^!$, then $f(x)\in C^!$, and if $x\in A^?$, then $f(x)\in C^?$. We leave this to the reader.

\smallskip

Case 2: If $\rlm(B)$ is finite, then so is $\rlm(A)$. The sets $A_\vt$, $A_\vb$, $A_\vn$, $B_\vt$, $B_\vb$, and $B_\vn$ are finite and have classical natural numbers associated with them. Let's call them $a_\vt$, $a_\vb$, $a_\vn$, $b_\vt$, $b_\vb$, and $b_\vn$, respectively. We have the following inequalities:
\begin{align*}
    a_\vt+a_\vb+a_\vn &\leq b_\vt+ b_\vb+b_\vn\\
    a_\vt+a_\vb &\leq b_\vt+ b_\vb\\
    a_\vt+a_\vn &\leq b_\vt+ b_\vn\\
    a_\vt&\leq b_\vt
\end{align*}
With a little finesse, we get
$$\mathrm{max}\{(a_\vb-b_\vb),0\}+\mathrm{max}\{(a_\vn-b_\vn),0\}\leq b_\vt-a_\vt.$$
We can now construct our counting $f$ of $A$ by sending each element of $A_\vt$ to an element in $B_\vt$. Then $b_\vt-a_\vt$ elements remain in $B_\vt$. We then send as many elements from $A_\vb$ into $B_\vb$ as we can. The remainder we send to what is left in $B_\vt$. There are now $b_\vt -(a_\vt+\mathrm{max}\{(a_\vb-b_\vb),0\})$ elements remaining in $B_\vt$. Finally, we send as many elements from $A_\vn$ to $B_\vn$ as possible, and the rest we send to what remains in $B_\vt$. This way, $f$ is an injection from $\rlm(A)$ with $f[A]\subseteq B$.

\smallskip

Case 3: Suppose that $B_\vt$ and $B_\vb$ are finite but $B_\vn$ is infinite. Then $A_\vn\preceq B_\vn$ since $A^?\preceq B^?$. Moreover, $a_\vt+a_\vb\leq b_\vt+b_\vb$ still holds, so $a_\vb-b_\vb\leq b_\vt-a_\vt$. We can therefore adjust our construction from the previous case to send all the elements from $A_\vn$ to $B_\vn$.

\smallskip

Case 4: $B_\vt$ and $B_\vn$ are finite but $B_\vb$ is infinite. This is similar to case 3.

\smallskip

Case 5: $B_\vt$ is finite, but $B_\vb$ and $B_\vn$ are infinite. Then, $A_\vt\preceq B_\vt$, $A_\vb\preceq B_\vb$, and $A_\vn\preceq B_\vn$. We can then construct the desired counting $f$ in the obvious way.

\medskip

\noindent\textbf{2)} Assume that $A\not\preceq B$, and let $f$ be a counting of $A^!$. We let $f'$ be any injective extension of $f$ to $\rlm(A)$. Then $f'[A]\not\subseteq B$, and therefore $f'[A^!]\not\subseteq B^?$. Now, $f[A^!]=f'[A^!]\not\subseteq B^?$.

Next, assume that $A^!\not\subseteq B^?$ and let $f$ be a counting of $A$. Then, the restriction of $f$ to $A^!$ is a counting of $A^!$. Thus, $f[A^!]\not\subseteq B^?$, and therefore $f[A]\not\subseteq B$.

\medskip

\noindent\textbf{3)} Assume that $A\cong B$. Then there is a counting $f$ of $A$ such that $f[A]=B$. It follows that $f[A_\vt]=B_\vt$, $f[A_\vb]=B_\vb$, and $f[A_\vn]=B_\vn.$ Thus, $A_\vt\cong B_\vt$, $A_\vb\cong B_\vb$, and $A_\vn\cong B_\vn$.

Now assume that $A_\vt\cong B_\vt$, $A_\vb\cong B_\vb$, and $A_\vn\cong B_\vn$. Then there are injections $g$, $h$, and $j$ such that $g[A_\vt]=B_\vt$, $h[A_\vb]=B_\vb$, and $j[A_\vn]=B_\vn$. We let
$$ f(x) = 
  \begin{cases}
    g(x), & \text{for } x\in A_\vt \\
    h(x), & \text{for } x\in A_\vb \\
    j(x), & \text{for }x\in A_\vn
  \end{cases}$$
and get $f[A]=B.$

\medskip

\noindent\textbf{4)} We will show that
$$\neg(A\ncong B)\text{ iff }A^!\preceq B^?\text{ and } B^!\preceq A^?\text{ and } A_\vb \preceq B_\vn\text{ and }B_\vb\preceq A_\vn.$$
First, we note that
$$\neg(A\ncong B)\text{ iff } \exists f\in\mathrm{Count}(A)(f[A^!]\subseteq B^?\land B^!\subseteq f[A^?]).$$
And, by symmetry, we also have
$$\neg(A\ncong B)\text{ iff } \exists g\in \mathrm{Count}(B)(g[B^!]\subseteq A^?\land A^!\subseteq g[B^?]).$$

Assume that $\neg(A\ncong B)$. Then, there is a counting $f$ of $A$ such that $f[A^!]\subseteq B^?$ and $B^!\subseteq f[A^?]$. Clearly, this gives $A^!\preceq B^?\text{ and } B^!\preceq A^?$. Next, we note that $B^!\subseteq f[A^?]$ and $f[A_\vb]\cap f[A^?]=\emptyset$, so $f[A_\vb]\cap B^!=\emptyset$. Since $f[A_\vb]\subseteq B^?$, we get $f[A_\vb]\subseteq B_\vn$. Thus, $A_\vb \preceq B_\vn$, and a similar argument gives $B_\vb\preceq A_\vn$.

\smallskip

Now let us assume that $A^!\preceq B^?$, $B^!\preceq A^?$, $A_\vb \preceq B_\vn,\text{ and }B_\vb\preceq A_\vn.$ We aim to find a counting $f$ of $A$ such that $f[A^!]\subseteq B^?$ and $B^!\subseteq f[A^?]$ or a counting $g$ of $B$ such that $g[B^!]\subseteq A^?$ and $A^!\subseteq g[B^?]$. We can w.l.o.g. assume that $\rlm(A)\cap \rlm(B)=\emptyset$. Using the axiom of choice, together with the fact that $A_\vt$ and $B_\vt$ are classical sets, we have three cases to consider: $A_\vt \cong B_\vt$, $A_\vt\preceq B_\vt$ and $A_\vt \not\cong B_\vt$, or $B_\vt\preceq A_\vt$ and $A_\vt \not\cong B_\vt$. 

Case 1: If $A_\vt\cong B_\vt$, then there is a bijection $f_1:A_\vt\rightarrow B_\vt$, and injections $f_2:A_\vb\rightarrow B_\vn$ and $g:B_\vb\rightarrow A_\vn$. We define our desired counting $f$ of $A$ by letting
$$ f(x) = 
  \begin{cases}
    f_1(x), & \text{for } x\in A_\vt \\
    f_2(x), & \text{for } x\in A_\vb \\
    g^{-1}(x), & \text{for }x\in g[B_\vb]\\
    x, & \text{for }x\in  A_\vn\setminus g[B_\vb].
  \end{cases}$$
It is then straightforward to verify that $f[A^!]\subseteq B^?$ and $B^!\subseteq f[A^?]$.

Case 2: If $A_\vt\preceq B_\vt$ and $A_\vt \not\cong B_\vt$, then there are a pair of injections $f_1:A_\vt\rightarrow B_\vt$ and $f_2: A_\vb\rightarrow B_\vn$. Moreover, since $B_\vt\cup B_\vb\preceq A_\vt\cup A_\vn$, we can do an exercise in classical set theory and find that $ (B_\vt\setminus f_1[A_\vt])\cup B_\vb \preceq A_\vn$. We let $g$ be the corresponding injection. We can now define our counting of $A$ by letting
$$ f(x) = 
  \begin{cases}
    f_1(x), & \text{for } x\in A_\vt \\
    f_2(x), & \text{for } x\in A_\vb \\
    g^{-1}(x), & \text{for }x\in g[B^!]\\
    x, & \text{for }x\in  A_\vn\setminus g[B^!].
  \end{cases}$$
Again, it is then straightforward to verify that $f[A^!]\subseteq B^?$ and $B^!\subseteq f[A^?]$.

Case 3: This is identical to case 2.

\end{proof}

\begin{definition}
    We call a set \emph{finite} if its realm is finite.
\end{definition}

For finite sets, we have a version of the
Schröder–Bernstein theorem.
\begin{proposition}\label{FinShrBern}
Let $A$ and $B$ be finite sets. If $A\preceq B$ and $B\preceq A$, then $A\cong B.$
\end{proposition}
\begin{proof}
    The sets $A_\vt$, $A_\vb$, $A_\vn$, $B_\vt$, $B_\vb$, and $B_\vn$ are finite and have classical natural numbers associated with them: $a_\vt$, $a_\vb$, $a_\vn$, $b_\vt$, $b_\vb$, and $b_\vn$, respectively. If $A\preceq B$ and $B\preceq A$, then we get the following equalities:
\begin{align*}
    a_\vt+a_\vb+a_\vn &= b_\vt+ b_\vb+b_\vn\\
    a_\vt+a_\vb &= b_\vt+ b_\vb\\
    a_\vt+a_\vn &= b_\vt+ b_\vn\\
    a_\vt&= b_\vt
\end{align*}
The last three give $a_\vt=b_\vt$, $a_\vb=b_\vb$, and $a_\vn=b_\vn$. Thus, $A\cong B$.

\end{proof}

\section{The cardinal numbers in $\BZFC$}\label{CardNumSect}
In this section, we give our formal definition of a cardinal number, establish some basic results about them, and give a treatment of cardinal arithmetic.


\subsection{Definition and basic properties}\label{FormalNum}
Our aim is to give an encoding of the cardinal numbers such that the following holds: 
\begin{equation}\label{CardEq}
|A|=|B|\Leftrightarrow A\cong B.
\end{equation}

A first guess would be to try to pick canonical representatives from each equivalence class, as is usually done in classical $\ZFC$. To see why this will not work in our setting, consider the set $A=\langle\{a\}|\{b\}|\{c\}\rangle$. We have $A\neq A$ since $A_\vb$ is non-empty. However, Theorem \ref{short} tells us that $\neg (A\not\cong A)$. Similarly, for any $X$ with $X\cong A$, we have $X\neq X$ and $\neg(X\not \cong X)$. Thus, no representative of the equivalence class of $A$ can serve as the cardinal number of $A$. Instead, we land on the following definition.

\begin{definition}\label{DefCardNum}
Given a set $A$, we define the \textit{cardinal number $|A|$ of $A$} as the triple $(X,Y,Z)$ where\footnote{Recall that for a classical set $B$, $|B|_{Cl}$ is the usual least von Neumann ordinal in one-to-one correspondence with $B$.}
$$X:=\boldsymbol{\langle} |A^!|_{Cl} \boldsymbol{\mid} |A^?|_{Cl}\boldsymbol{\rangle},\quad 
Y:=\boldsymbol{\langle} |A_\vt|_{Cl}\boldsymbol{\mid}|A^?|_{Cl}\boldsymbol{\rangle},\quad\text{and}\quad
Z:=\boldsymbol{\langle} |A_\vb|_{Cl}\boldsymbol{\mid}|A_\vn|_{Cl}\boldsymbol{\rangle}.$$
By a \textit{cardinal number} we mean an element from the class $\{|x|:x\in V\}.$ 
\end{definition}

\begin{theorem}\label{EqCardThrm}
For all $A$ and $B$,
$$|A|=|B|\Leftrightarrow A\cong B.$$
\end{theorem}
\begin{proof}
Let $|A|=(X_1,X_2,X_3)$ and $|B|=(Y_1,Y_2,Y_3)$. We have
\begin{align*}
|A|=|B|&\Leftrightarrow X_1=Y_1\land X_2=Y_2\land X_3=Y_3\\
&\leftrightarrow (|A^!|_{Cl}=|B^!|_{Cl}\land |A^?|_{Cl}=|B^?|_{Cl})\\
& \ \ \ \land (|A_\vt|_{Cl}=|B_\vt|_{Cl}\land |A^?|_{Cl}=|B^?|_{Cl})\\
& \ \ \ \land (|A_\vb|_{Cl}=|B_\vb|_{Cl}\land |A_\vn|_{Cl}=|B_\vn|_{Cl})\\
&\leftrightarrow |A_\vt|_{Cl}=|B_\vt|_{Cl}\land |A_\vb|_{Cl}=|B_\vb|_{Cl}\land |A_\vn|_{Cl}=|B_\vn|_{Cl}\\
&\leftrightarrow A_\vt\cong B_\vt\land A_\vb\cong B_\vb\land A_\vn\cong B_\vn\\
&\leftrightarrow A\cong B.
\end{align*}

For the remainder of the proof, it is crucial to remember that for classical sets $X$ and $Y,$ $|X|_{Cl}\leq |Y|_{Cl}$ iff $|X|_{Cl}\subseteq |Y|_{Cl}$. So, $|X|_{Cl}\nleq |Y|_{Cl}$ iff $\exists \alpha (\alpha \in |X|_{Cl}\land \alpha\notin|Y|_{Cl}).$

\smallskip

Now, assume that $A\ncong B$. Then one of the following holds: $|A^!|_{Cl}\nleq |B^?|_{Cl}$,  $|B^!|_{Cl}\nleq |A^?|_{Cl}$, $|A_\vb|_{Cl}\nleq |B_\vn|_{Cl}$, or $|B_\vb|_{Cl}\nleq |A_\vn|_{Cl}$. 

If $|A^!|_{Cl}\nleq |B^?|_{Cl}$ or $|B^!|_{Cl}\nleq |A^?|_{Cl}$, then $X_1\neq Y_1$, and therefore $|A|\neq|B|$. 

If $|A_\vb|_{Cl}\nleq |B_\vn|_{Cl}$ or $|B_\vb|_{Cl}\nleq |A_\vn|_{Cl}$, then $X_3\neq Y_3$, and therefore $|A|\neq |B|$. 

\smallskip

Finally, suppose that $|A|\neq |B|$, then $X_1\neq Y_1$ or $X_2\neq Y_2$ or $X_3\neq Y_3$. 

If $X_1\neq Y_1$, then $\exists \alpha (\alpha\in |A^!|_{Cl}\land \alpha \notin |B^?|_{Cl})$ or $\exists \alpha (\alpha\in |B^!|_{Cl}\land \alpha \notin |A^?|_{Cl})$. In either case, $A^!\npreceq B^?\text{ or } B^!\npreceq A^?$, and therefore $A\ncong B$.

If $X_2\neq Y_2$, then $\exists \alpha (\alpha\in |A_\vt|_{Cl}\land \alpha \notin |B^?|_{Cl})$ or $\exists \alpha (\alpha\in |B_\vt|_{Cl}\land \alpha \notin |A^?|_{Cl})$, and therefore $\exists \alpha (\alpha\in |A^!|_{Cl}\land \alpha \notin |B^?|_{Cl})$ or $\exists \alpha (\alpha\in |B^!|_{Cl}\land \alpha \notin |A^?|_{Cl})$. Thus, $A\ncong B$.

If $X_3\neq Y_3$, then $\exists \alpha (\alpha\in |A_\vb|_{Cl}\land \alpha \notin |B_\vn|_{Cl})$ or $\exists \alpha (\alpha\in |B_\vb|_{Cl}\land \alpha \notin |A_\vn|_{Cl}),$ so $A_\vb \npreceq B_\vn\text{ or }B_\vb\npreceq A_\vn$, and therefore $A\ncong B$.

\end{proof}

Theorem \ref{long} allows us to make the following definition.

\begin{definition}
We order the cardinal numbers by letting 
$$|A|\leq |B|:\Leftrightarrow A\preceq B.$$
\end{definition}

\medskip

We easily get the following theorems.

\begin{theorem}\label{longcard}
For all $A$, $B$ and $C$,
\begin{enumerate}
\item $|A|=|A|$;
\item $|A|=|B|\Leftrightarrow |B|=|A|$;
\item if $|A|=|B|$, then $|A|=|C|\Leftrightarrow |B|=|C|$;
\item if $|A|=|B|$, then $|A|\leq |C|\Leftrightarrow |B|\leq |C|$;
\item if $|A|=|B|$, then $|C|\leq |A|\Leftrightarrow |C|\leq |B|$;
\item if $|A|\leq |B|$ and $|B|\leq |C|$, then $|A|\leq |C|$.
\end{enumerate}
\end{theorem}
\begin{proof}
Follows immediately from Theorem \ref{long}.
\end{proof}

\begin{theorem}\label{shortcard}
    For all $A$ and $B$,
{\small
\begin{enumerate}\setlength{\itemindent}{-1em}
    \item $|A| \leq |B|\;\text{ iff }\; {|\rlm(A)|\leq|\rlm(B)|\text{ and } |A^!|\leq |B^!|\text{ and }|A^?|\leq |B^?|\text{ and }|A_\vt| \leq |B_\vt|}$;
    \item $|A| \not\leq |B|\;\text{ iff }\; |A^!|\not\leq |B^?|$;
    \vspace{0.5em}
    \item $|A| =|B|\;\text{ iff }\; |A_\vt|=|B_\vt|\text{ and } |A_\vb|=|B_\vb|\text{ and }|A_\vn|=|B_\vn|$;
    \item $|A|\neq|B|\;\text{ iff }\; |A^!|\nleq |B^?|\text{ or } |B^!|\nleq  |A^?|\text{ or } |A_\vb| \nleq |B_\vn|\text{ or }|B_\vb|\nleq |A_\vn|$.
\end{enumerate}
}
\end{theorem}
\begin{proof}
    Follows from Theorem \ref{short}.
\end{proof}

\subsection{Cardinal arithmetic}
The cardinal numbers come with a rich arithmetic. 

\begin{definition}
Let $A$ and $B$ be sets. We let
$$|A|+|B|:=|A\uplus B|\text{ and }|A|\cdot|B|:=|A\times B|.$$
\end{definition}
\medskip

We get the usual rules for associativity, commutativity, and distributivity.
\begin{proposition}
    For all cardinal numbers $\kappa$, $\mu$, and $\lambda$,

    \begin{enumerate}
        \item $\kappa+(\mu+\lambda)=(\kappa+\mu)+\lambda$;
        \item $\kappa(\mu\lambda)=(\kappa\mu)\lambda$;
        \item $\kappa+\mu=\mu+\kappa$;
        \item $\kappa\mu=\mu\kappa$;
        \item $\kappa(\mu+\lambda)=\kappa\mu+\kappa\lambda$.
    \end{enumerate}
\end{proposition}

\begin{definition}
We introduce the following cardinal numbers:
$$0:=|\emptyset|,\;1:=|\{\emptyset\}|,
\;\bb:=|\langle \{\emptyset\}\mid \emptyset\mid\emptyset\rangle|, \:\text{and}\:\nn:=|\langle \emptyset\mid \emptyset\mid\{\emptyset\}\rangle|.$$
\end{definition}

\begin{figure}[H]
    \centering
\begin{tikzpicture}[line cap=round,line join=round,scale=0.3]
\clip (-3.8,-2.3) rectangle (3.8,2.3);
\draw [line width=0.6pt] (-1,0) circle (2);
\draw [line width=0.6pt] (1,0) circle (2);
\end{tikzpicture}
\begin{tikzpicture}[line cap=round,line join=round,scale=0.3]
\clip (-3.8,-2.3) rectangle (3.8,2.3);
\draw [line width=0.6pt] (-1,0) circle (2);
\draw [line width=0.6pt] (1,0) circle (2);
\node[font=\footnotesize] at (0,0) {$\bullet$};
\end{tikzpicture}
\begin{tikzpicture}[line cap=round,line join=round,scale=0.3]
\clip (-3.8,-2.3) rectangle (3.8,2.3);
\draw [line width=0.6pt] (-1,0) circle (2);
\draw [line width=0.6pt] (1,0) circle (2);
\node[font=\footnotesize] at (-1.75,0) {$\bullet$};
\end{tikzpicture}
\begin{tikzpicture}[line cap=round,line join=round,scale=0.3]
\clip (-3.8,-2.3) rectangle (3.8,2.3);
\draw [line width=0.6pt] (-1,0) circle (2);
\draw [line width=0.6pt] (1,0) circle (2);
\node[font=\footnotesize] at (1.75,0) {$\bullet$};
\end{tikzpicture}
\caption{The cardinal numbers $0$, $1$, $\bb$, and $\nn$, respectively.}
\label{fig:constants}
\end{figure}

\medskip

The following theorem tells us that every cardinal number can be expressed as a linear combination of the numbers 1, $\bb$, and $\nn$ with classical cardinals as coefficients.
\begin{theorem}
For all $A$,
$$|A|=|A_\vt|+|A_\vb|\cdot \bb+|A_\vn|\cdot \nn.$$
\end{theorem}
\begin{proof}
We have
\begin{align*}
    |A_\vt|+|A_\vb|&\cdot \bb+|A_\vn|\cdot \nn=|A_\vt \uplus (A_\vb\times \langle \{\emptyset\}\mid\emptyset\rangle)  \uplus (A_\vn\times \langle \emptyset\mid\{\emptyset\}\rangle ) |\\[8pt]
    &=|\big\langle\{(x,0):x\in A_\vt\}\mid\{(x,1):x\in A_\vb\}\mid \{(x,2):x\in A_\vn\}\big\rangle|
\end{align*}

For $x\in \rlm (A)$, let
$$ f(x) = 
  \begin{cases}
    (x,0), & \text{for } x\in A_\vt \\
    (x,1), & \text{for } x\in A_\vb \\
    (x,2), & \text{for }x\in A_\vn
  \end{cases}$$
Now,
$$f[A]=\big\langle\{(x,0):x\in A_\vt\}\mid\{(x,1):x\in A_\vb\}\mid \{(x,2):x\in A_\vn\}\big\rangle.$$
\end{proof}

\begin{corollary}
    For each cardinal number $\kappa$, there are uniquely determined classical cardinals $\kappa_\vt$, $\kappa_\vb$, and $\kappa_\vn$ such that $$\kappa=\kappa_\vt+\kappa_\vb\cdot\bb+\kappa_\vn\cdot\nn.$$
\end{corollary}


\smallskip

The following theorem tells us how addition and multiplication of non-classical cardinals can be computed using the classical cardinals.

\begin{theorem}\label{addandmultcard}
For all cardinal numbers $\kappa$ and $\mu$, 
$$\kappa+\mu=(\kappa_{\vt}+\mu_{\vt})+(\kappa_{\vb}+\mu_{\vb})\cdot\mathfrak{b}+(\kappa_{\vn}+\mu_{\vn})\cdot\mathfrak{n}$$
       and
       $$\kappa\cdot \mu=\kappa_{\vt} \mu_{\vt}+(\kappa_{\vt}\mu_{\vb}+\kappa_{\vb}\mu_{\vt}+\kappa_{\vb}\mu_{\vb})\cdot\mathfrak{b}+(\kappa_{\vt}\mu_{\vn}+\kappa_{\vn}\mu_{\vt}+\kappa_{\vn}\mu_{\vn})\cdot\mathfrak{n}.$$
Moreover, 
$$\kappa \cdot 1 = \kappa,\:\kappa \cdot 0 = 0,\:\bb^2 = \bb, \:\nn^2 = \nn, \:\text{and }\:\bb\cdot\nn = 0.$$
\end{theorem}
\smallskip
\begin{proof}
    We only show that $\bb\cdot\nn=0$ and that $\kappa\cdot \mu=\kappa_{\vt} \mu_{\vt}+(\kappa_{\vt}\mu_{\vb}+\kappa_{\vb}\mu_{\vt}+\kappa_{\vb}\mu_{\vb})\cdot\mathfrak{b}+(\kappa_{\vt}\mu_{\vn}+\kappa_{\vn}\mu_{\vt}+\kappa_{\vn}\mu_{\vn})\cdot\mathfrak{n}.$ 

    \smallskip

    $\bb\cdot\nn=0$: First, notice that for any $A$ and $B$, $(A\times B)^!=A^!\times B^!$ and $(A\times B)^?=A^?\times B^?$. Moreover, $$\langle \{\emptyset\}\mid \emptyset\mid\emptyset\rangle=\langle \{\emptyset\}\mid\emptyset\rangle\quad\text{ and }\quad \langle \emptyset\mid \emptyset\mid\{\emptyset\}\rangle=\langle \emptyset\mid\{\emptyset\}\rangle.$$
    Thus,
    \begin{align*}
        \bb\cdot \nn&=|\langle \{\emptyset\}\mid \emptyset\mid\emptyset\rangle\times\langle \emptyset\mid \emptyset\mid\{\emptyset\}\rangle|\\
        &=|\langle \{\emptyset\}\mid \emptyset\rangle\times\langle \emptyset\mid \{\emptyset\}\rangle|\\
        &=|\langle \emptyset\mid \emptyset\rangle|\\
        &=|\emptyset|\\
        &=0.
    \end{align*}

   For the second point, we use  associativity, commutativity, and distributivity to get that $\kappa\cdot \mu$ is equal to 
\begin{align*}
    \kappa_{\vt} \mu_{\vt}+&(\kappa_{\vt}\mu_{\vb}+\kappa_{\vb}\mu_{\vt})\cdot\mathfrak{b}+ (\kappa_\vb \mu_\vb)\cdot \bb^2\\
    &+(\kappa_{\vt}\mu_{\vn}+\kappa_{\vn}\mu_{\vt})\cdot\mathfrak{n}+(\kappa_\vn \mu_\vn)\cdot \nn^2\\
    &\quad +(\kappa_\vb \mu_\vn+\kappa_\vn \mu_\vb) \cdot \bb\cdot \nn
\end{align*}
Using the identities for $\bb$ and $\nn$, we now get
$$\kappa_{\vt} \mu_{\vt}+(\kappa_{\vt}\mu_{\vb}+\kappa_{\vb}\mu_{\vt}+\kappa_{\vb}\mu_{\vb})\cdot\mathfrak{b}+(\kappa_{\vt}\mu_{\vn}+\kappa_{\vn}\mu_{\vt}+\kappa_{\vn}\mu_{\vn})\cdot\mathfrak{n}.$$
\end{proof}

\begin{theorem}\label{comparecardnum} For all cardinals $\kappa$ and $\mu$,
\begin{enumerate}
\item $\kappa \leq \mu$ \quad iff \quad the following conditions are satisfied:
\begin{align*}
\kappa_\vt + \kappa_\vb + \kappa_\vn &\leq \mu_\vt + \mu_\vb + \mu_\vn, \\
\kappa_\vt + \kappa_\vb &\leq \mu_\vt + \mu_\vb, \\
\kappa_\vt + \kappa_\vn &\leq \mu_\vt + \mu_\vn,\text{ and} \\
\kappa_\vt &\leq \mu_\vt\text{;}
\end{align*}
\item $\kappa \not\leq \mu$ \quad iff \quad $\kappa_\vt + \kappa_\vb \not\leq \mu_\vt + \mu_\vn$;
\vspace{0.5em}
\item $\kappa = \mu$ \quad iff \quad $\kappa_\vt = \mu_\vt$, $\kappa_\vb = \mu_\vb$, and $\kappa_\vn = \mu_\vn$;
\item $\kappa \neq \mu$ \quad iff \quad at least one of the following holds:
\begin{align*}
\kappa_\vt + \kappa_\vb &\not\leq \mu_\vt + \mu_\vn, \\
\mu_\vt + \mu_\vb &\not\leq \kappa_\vt + \kappa_\vn, \\
\kappa_\vb &\not\leq \mu_\vn,\text{ or} \\
\mu_\vb &\not\leq \kappa_\vn.
\end{align*}
\end{enumerate}
\end{theorem}
\begin{proof}
    This follows from Theorem \ref{shortcard}.
    
\end{proof}

It should be pointed out that $\leq$ fails to be a partial order on the cardinal numbers, as we have both
$$\aleph_0\leq\aleph_0+\bb\quad\text{ and }\quad\aleph_0+\bb\leq \aleph_0.$$
But $$\neg(\aleph_0=\aleph_0+\bb).$$ 
This simply tells us that there is a set with $\aleph_0$ elements that is a subset of a set $\aleph_0+\bb$, and there is a set with $\aleph_0+\bb$ that is a subset of a set with $\aleph_0$. However, there is no set that has both $\aleph_0+\bb$ and $\aleph_0$ elements.

\begin{definition}\label{integers}   
We call a cardinal number $\kappa$ \emph{finite} if the classical cardinals $\kappa_\vt$, $\kappa_\vb$, and $\kappa_\vn$ are all finite. We will refer to finite cardinal numbers simply as \emph{finite numbers}.
\end{definition}

\begin{proposition}
Let $n$ and $m$ be finite numbers. If $n\leq m$ and $m\leq n $, then $n=m.$
\end{proposition}
\begin{proof}
    This follows from Proposition~\ref{FinShrBern}.
    
\end{proof}

To increase the size of a finite number, we can do one of two things: We can add elements to its classical, inconsistent, or incomplete parts, or we can move elements from one of its non-classical parts to its classical part.

We can therefore imagine them on a three-dimensional grid. One direction for the classical part, one direction for the inconsistent part, and one direction for the incomplete part. One number is larger than another if it can be reached using only two kinds of steps. The first kind simply moves positively along any coordinate axis. The second kind moves diagonally, decreasing either the inconsistent or incomplete coordinate, while increasing the classical coordinate. (See Figure \ref{steps}.)

\begin{figure}[H]
\centering


\begin{tikzcd}
                                      & \quad 1+\nn\quad \arrow[rrr]                                         &  &                                             & 1+\bb+\nn                                                 \\
\quad 1 \: \arrow[ru] \arrow[rrr]              &                                                           &  & 1+\bb \arrow[ru]                            &                                                           \\
                                      &                                                           &  &                                             &                                                           \\
                                      &                                                           &  &                                             &                                                           \\
                                      & \:\nn \: \arrow[rrr, dashed] \arrow[uuuu, dashed] \arrow[luuu] &  &                                             & \bb+\nn \arrow[uuuu] \arrow[luuu] \arrow[llluuuu, dashed] \\
0 \arrow[uuuu] \arrow[rrr] \arrow[ru] &                                                           &  & \bb \arrow[uuuu] \arrow[ru] \arrow[llluuuu] &                                                          
\end{tikzcd}
\caption{A diagram showing the first few cardinal numbers. Some lines are dashed for visual clarity.}
\label{steps}
\end{figure}
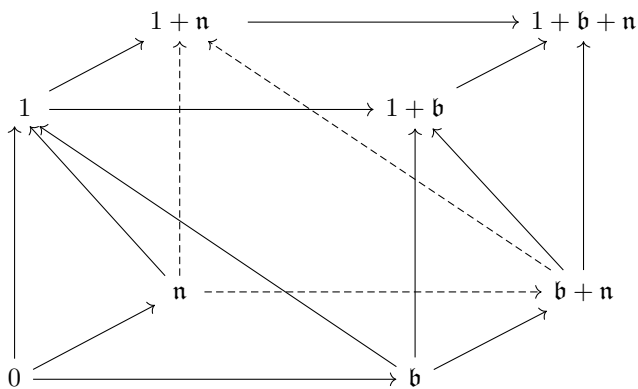

\section{Conclusion}
In this paper, we have presented a thorough treatment of cardinality in the paraconsistent and paracomplete set theory $\BZFC$. We have given an account of what it means for two non-classical sets to have the same size, and we have constructed the cardinal numbers that tell us how many elements a set has. 

In the introduction, we asked: how many elements does the inconsistent set $A = \langle\{a\}|\emptyset|\emptyset\rangle$ have? We now have a clear answer. Just as a singleton has \emph{one} element, the set $A$ has exactly $\bb$ elements.  Similarly, the incomplete set $\langle\emptyset|\emptyset|\{a\}\rangle$ has $\nn$ elements. These two cardinals turned out to be all we needed to express every non-classical cardinal, as we showed that any cardinal number can be expressed as a linear combination of 1, $\bb$, and $\nn$ over the classical cardinals.

\smallskip

\smallskip

We conclude by posing an open question: \emph{Is there a natural generalization to the real numbers?} It seems clear that we can simply write $x_\vt+x_\vb\cdot \bb+x_\vn\cdot\nn$ for any classical real numbers $x_\vt$, $x_\vb$, $x_\vn$. We can add and multiply these numbers using the rules from Theorem \ref{addandmultcard}. We can define subtraction as $x+(-y)$, where 
$$-y:=-y_\vt+(-y_\vb)\cdot\bb+(-y_\vn)\cdot \nn.$$
Similarly, if $y_\vt\neq 0$, $y_\vt+y_\vb\neq 0$, and $y_\vt+y_\vn\neq 0$, we can define division as $x\cdot y^{-1}$, where
$$y^{-1}:=\frac{1}{y_\vt}+\frac{-y_\vb}{y_\vt(y_\vt+y_\vb)}\cdot \bb+\frac{-y_\vn}{y_\vt(y_\vt+y_\vn)}\cdot \nn.$$
In this way, $x-x=0$ and $x\cdot x^{-1}=1$.

What is not clear is how to interpret $x\neq y$, $x\leq y$, and $x\nleq y$. Is there a reason to believe that the equivalences of Theorem \ref{comparecardnum} should hold here? We would have to believe $-2\nn\neq -\bb$, but this seems counterintuitive, as $\neg(2\nn\neq\bb)$. What is needed is a convincing account of how these numbers are to be understood.

\bibliographystyle{alpha}
\bibliography{references}

\end{document}